\documentclass[a4paper,fleqn]{cas-sc}

\usepackage[numbers]{natbib}

\usepackage{graphicx} 
\usepackage{xcolor}
\usepackage{tikz}
\usepackage{esint}
\usepackage{amsmath}

\def\Yint#1{\mathchoice
{\YYint\displaystyle\textstyle{#1}}%
{\YYint\textstyle\scriptstyle{#1}}%
{\YYint\scriptstyle\scriptscriptstyle{#1}}%
{\YYint\scriptscriptstyle\scriptscriptstyle{#1}}%
\!\iint}

\def\YYint#1#2#3{{%
\setbox0=\hbox{$#1{#2#3}{\iint}$}%
\vcenter{\hbox{$#2#3$}}\kern-.50\wd0}}

\def\longdash{-\mkern-9.5mu-}
\def\tiltlongdash{\rotatebox[origin=c]{18}{$\longdash$}}

\def\fiint{\Yint\tiltlongdash}

\newtheorem{theorem}{Theorem}[section]
\newtheorem{lemma}[theorem]{Lemma}
\newtheorem{proposition}[theorem]{Proposition}
\newtheorem{definition}[theorem]{Definition}
\newtheorem{corollary}[theorem]{Corollary}
\newtheorem{remark}[theorem]{Remark}

\newproof{proof}{Proof}

\begin{document}
\let\WriteBookmarks\relax
\def\floatpagepagefraction{1}
\def\textpagefraction{.001}

\shorttitle{Interior regularity of parabolic mixed local-nonlocal equations}    

\shortauthors{S. Junquera and L. M. Del Pezzo}  

\title [mode = title]{Interior regularity of mixed local-nonlocal parabolic semilinear equations}  
%

\author[1]{Sergio Junquera}[orcid=0009-0004-2759-2752]

\cormark[1]

\ead{sejunque@ucm.es}

\credit{Conceptualization, Methodology, Formal analysis, Writing - Original Draft, Writing - Review and Editing}

\affiliation[1]{organization={Departamento de Análisis Matemático y Matemática Aplicada, Universidad Complutense de Madrid},
			addressline={Plaza de Ciencias 3, Ciudad Universitaria},
            city={Madrid},
            postcode={28040}, 
            state={Madrid},
            country={Spain}}

\author[2,3]{Leandro M. Del Pezzo}

\ead{leandro.delpezzo@fcea.edu.uy}

\credit{Conceptualization, Methodology, Formal analysis, Writing - Review and Editing}

\affiliation[2]{organization={Instituto de Estadística (IESTA), Departamento de Métodos Cuantitativos, Facultad de Ciencias Económicas y de Administración, Universidad de la República},
            addressline={Av. Gonzalo Ramírez 1926}, 
            city={Montevideo},
            postcode={11200}, 
            state={Montevideo},
            country={Uruguay}}

\affiliation[3]{organization={PEDECIBA Matemática},
            addressline={Iguá 4225}, 
            city={Montevideo},
            postcode={11400}, 
            state={Montevideo},
            country={Uruguay}}

\cortext[1]{Corresponding author}

\begin{abstract}
In this paper we prove the existence, uniqueness and regularity of classical solutions \break of a semilinear parabolic equation with a mixed local and nonlocal diffusion operator \break $\mathcal{L} = -(-\Delta)^s + \Delta$ and Dirichlet boundary conditions. Here, $(-\Delta)^s$ is the integral fractional laplacian and $\Delta$ is the classic local laplacian. We then study the interior regularity of said solutions and conclude that they are Hölder continuous in both space and time, and they are $C^{2,\alpha}_{loc}$ in space for all positive times.
\end{abstract}


\begin{keywords}
Nonlocal diffusion \sep Fractional laplacian \sep Mixed operator \sep Interior regularity \sep Hölder continuity
\end{keywords}

\maketitle

\section{Introduction} \label{sec.introduction}

In this paper we will study the interior regularity of solutions of the following nonlinear nonlocal parabolic dynamical system with Dirichlet boundary conditions:
\begin{equation} \label{1.1}
    \begin{cases}
        \displaystyle u_t(x,t)= \mathcal{L} u(x,t)+f(u(x,t)), & (x,t)\in\Omega\times(0,T), \\
        u(x,0)=u_0(x), & x\in\Omega,\\
        u(x,t) = 0, & (x,t)\in\mathbb{R}^n \setminus \Omega \times [0,T),
    \end{cases}
    \end{equation}
    where we consider that $T\in(0,\infty]$ is the maximal existence time of the solution and $\Omega \subset \mathbb{R}^n$ is an open bounded Lipschitz domain. The reaction term is a nonnegative function $f:\mathcal{I} \to \mathbb{R}^+$, where $\mathcal{I} = [0,\alpha)$ for some $\alpha>0$ and such that $f\in C^1_{loc}(\mathcal{I})$ and $f$ is increasing. The initial data $u_0$ lies in the space $H_0^1(\Omega) \cap L^\infty(\Omega)$ and it satisfies $u_0(x) \in \mathcal{I}$ a.e. $x\in\Omega$. The diffusion operator $\mathcal{L}$ is a combination of a nonlocal operator (the fractional laplacian) and a local operator (the classical laplacian),
    \begin{equation*}
    	\mathcal{L} = -(-\Delta)^s + \Delta,
    \end{equation*}
    with $s\in(0,1)$ and
    \begin{equation} \label{eq.fraclap}
    	(-\Delta)^s \psi(x) = C_{n,s} \text{P.V.} \int_{\mathbb{R}^n} \frac{\psi(x)-\psi(y)}{|x-y|^{n+2s}} dy,
    \end{equation}
    where $C_{n,s}$ is a normalization constant that depends on $n$ and $s$ and P.V. denotes the principal value integral. This principal value is used to avoid the singularity of the integrand at the diagonal $x=y$, and is defined as
	\begin{equation*}
		\text{P.V.} \int_{\mathbb{R}^n} F(x,y) dy = \lim_{\varepsilon\to 0^+} \int_{\mathbb{R}^n \setminus B_\varepsilon (x)} F(x,y) dy.
	\end{equation*}   
     This kind of operators generate high mathematical interest due to the combination of local and nonlocal mechanisms, as well as due to the lack of scale invariance, see for example \cite{AC2021,BDVV2023,SVWZ2025}. They are also highly interesting operators from an application viewpoint, being able to model diffusion phenomena with different time scales and bimodal power-law distribution processes, see \cite{PV2021}, among others. The main mechanism that differentiates the mixed operator from the purely local or purely fractional laplacian is the combination of different scales, which rises the interesting question of whether one of the scales prevail or both coexist in the end.
    
    In this paper we will work with Dirichlet boundary conditions, which must be imposed in the whole complement of the domain, $\mathbb{R}^n \setminus \Omega$. This is natural in the field of nonlocal operators, since their nature forces us to establish the value of the solution in the whole space $\mathbb{R}^n$, see \cite{AMRT2010, BFRW1997, BV2016, F2003}. In our case, this is needed so that expression \eqref{eq.fraclap} makes sense.
    
    Notice that the semilinear reaction term in \eqref{1.1} is a locally Lipschitz function with a restricted domain. This was inspired by the problem the authors strived to study in \cite{JP2026}, where
    \begin{equation*}
    	f(u(x,t)) = \lambda (1-u(x,t))^{-p}.
    \end{equation*}
    This reaction term explodes when $\|u(\cdot,t)\|_{L^\infty(\Omega)}\nearrow 1$ as $t\to T$. When this happens, the equation stops making sense and we say that the solution presents quenching, see for example \cite{L1985} for more information. Although this quenching system was our main objective and the results in this paper are formulated in a somewhat particular case that simplifies some arguments, those could also be easily adapted to the most general case in which $\mathcal{I}=(\alpha_1,\alpha_2)$ with a little bit more work.
    
   	The study of partial differential equations involving the mixed local and nonlocal operator $\mathcal{L}$ has been attracting considerable attention, mainly in the elliptic framework. Some of those studies deal with existence and non-existence of solutions \cite{AC2021, BDVV2025}, eigenvalue problems \cite{BDVV2023,DFR2019} or symmetry of solutions \cite{BVDV2021}, while others, the ones we are most interested in, study regularity of solutions, such as \cite{SVWZ2025}.
   	
   	In contrast with the elliptic case, parabolic evolution equations with this mixed local and nonlocal operator have not been studied as deeply, although some very interesting works already exist in this context. For example, in \cite{D2024} they study the interior $C^\alpha$ regularity of the homogeneous parabolic problem: in \cite{BPV2025} they deal with the existence and non-existence of solutions for the Cauchy problem with a superlinear reaction term, and in \cite{DF2025} they obtain the Fujita exponent and blow-up rate of the solutions for that same problem.
   	
   	The main purpose of this article is to establish interior $C^{2,\alpha}$ regularity for the semilinear mixed local-nonlocal parabolic equation \eqref{1.1}. We will also study and obtain a spatial $L^2-L^\infty$ estimate of the mixed local-nonlocal parabolic equation with a zeroth order term, namely
   	\begin{equation} \label{eq.zeroth.order}
    \begin{cases}
        \displaystyle v_t(x,t)= \mathcal{L} v(x,t)+a(x,t)v(x,t), & (x,t)\in\Omega\times(0,T^\prime), \\
        v(x,0)=v_0(x), & x\in\Omega,\\
        v(x,t) = 0, & (x,t)\in\mathbb{R}^n \setminus \Omega \times [0,T^\prime),
    \end{cases}
    \end{equation}
    where we consider that $T^\prime\in(0,\infty]$ is the maximal existence time of the solution and $\Omega \subset \mathbb{R}^n$ is an open bounded Lipschitz domain. The coefficient of the zeroth order term is a function $a\in C_{loc}([0,T^\prime), L^\infty(\Omega))$ and the initial data $v_0$ lies in the space $H_0^1(\Omega)$. This $L^2-L^\infty$ estimate will be a key ingredient on tackling the main objective of this paper.
    
    \subsection{Main results}
    
    Our first objective is to prove that system \eqref{1.1} admits classical solutions in the $L^2(\Omega)$ framework. We will do so via semigroup theory and find the following result
    
    \begin{proposition} \label{prop.semigroup.regularity}
    	For every $u_0\in H_0^1(\Omega) \cap L^\infty(\Omega)$ such that $\|u_0\|_{L^\infty(\Omega)} < \alpha$, there exists some $T=T(u_0)>0$ such that system \eqref{1.1} admits a unique classical solution $u \in C((0,T), D(\mathcal{L})) \cap C([0,T),H_0^1(\Omega)) \cap C^1((0,T),L^2(\Omega))$. 
    	
    	Furthermore, $\limsup_{t\to t_0} \|u(\cdot,t)\|_{L^\infty(\Omega)} = \alpha$ if and only if $t_0=T$.
    \end{proposition}
    
    Then, we will establish some initial interior regularity on the solution in both space and time.
    
    \begin{theorem} \label{teo.holder}
    	Let $u$ be the unique classical solution of system \eqref{1.1} and some $\tilde{T} \in (0,T)$. Fix some $(x_0,t_0) \in \Omega \times (0,\tilde{T})$ and define, for any $\rho>0$, the parabolic cylinder $Q_\rho = B_\rho(x_0) \times [t_0-\rho^2,t_0]$. Then, for every $\alpha_0 \in (0,1)$, there exists some $r_\ast = r_\ast(n,s,\alpha_0,\tilde{T}) \in (0,1)$ such that $\overline{B_\rho}(x_0) \subset \Omega$ and $[t_0-\rho^2,t_0] \subset (0,\tilde{T})$ and, for every $0<r\leq r_\ast$, it holds
    	\begin{equation*}
    		\fiint_{Q_r} |u(x,t)-(u)_{Q_r}|^2 dx dt \leq c r^{2\alpha_0},
    	\end{equation*}
    	where $c=c(n,s,\alpha_0,\|f(u)\|_{L^\infty(\Omega \times [0,\tilde{T}])}, \|u\|_{L^2(\Omega \times[0,\tilde{T})}, \|u\|_{L^\infty([0,\tilde{T}], L^2(\Omega))})$. In particular, the solution $u$ is locally Hölder continuous of exponent $\alpha_0$ in both space and time.
    \end{theorem}
    
    Next, we will obtain an $L^2-L^\infty$ estimate in space for system \eqref{eq.zeroth.order}, which will help us in the last stretch of the article.
    
    \begin{theorem} \label{teo.l2.linfty}
    	Let $v$ be the unique classical solution to \eqref{eq.zeroth.order} and fix some $\tilde{T} \in (0,T^\prime)$. Let $(x_0,t_0) \in \Omega \times (0,\tilde{T})$, $\tau \in (0,1)$, $R\in (0,1)$ and $Q_R = B_R(x_0) \times (t_0-R^2,t_0)$, such that $\overline{B_R}(x_0) \subset \Omega$ and $[t_0-R^2,t_0] \subset (0, \tilde{T})$. Then it holds that
    	\begin{equation} \label{eq.L2Linfty}
    		\|v\|_{L^\infty(Q_{\tau R})} \leq  C R^{-\frac{n}{2}} \tau^{-\frac{n}{2}-1} (1-\tau)^{-(\frac{n}{2}+s)(\frac{n}{2}+1)} \|v\|_{L^\infty([t_0-R^2,t_0], L^2(\Omega))},
    	\end{equation}
    	where $C=C(n,s,\|a\|_{L^\infty(\Omega\times[0,\tilde{T}])})$.
    \end{theorem}
    
    Finally, we will combine these results into the main objective of this article, which is to obtain higher interior regularity in space for solutions of \eqref{1.1}.
    
    \begin{theorem} \label{teo.higherregularity}
    	Let $u$ be the unique classical solution of system \eqref{1.1}. Then for every $\alpha_0 \in (0,1)$, $u_t \in C^{\alpha_0}_{loc}(\Omega \times (0,T))$ and $u(\cdot,t) \in C^{2,\alpha_0}_{loc}(\Omega)$ for every $t\in (0,T)$.
    \end{theorem}
    
    \noindent \textbf{The paper is structured as follows:} in Section \ref{sec.preliminary} we will introduce some relevant definitions and notation. Then, in Section \ref{sec.existence}, we will obtain the existence, uniqueness and some regularity of solutions to system \eqref{1.1} by proving Proposition \ref{prop.semigroup.regularity}. In order to do so, we will also establish some comparison results for the same system, see Lemma \ref{initial.data.comparison} and Lemma \ref{lemma.comparison.abstract}. In Section \ref{sec.campanato}, we prove Theorem \ref{teo.holder}, that is, the Hölder continuity in space and time of solutions of \eqref{1.1}. We do so via a Campanato argument. Then, in Section \ref{sec.l2linfty}, we prove Theorem \ref{teo.l2.linfty}, which gives us an $L^2-L^\infty$ estimate on space of solutions of \eqref{eq.zeroth.order}. This is done by obtaining a Caccioppoli inequality for the system, see Lemma \ref{lemma.Cacciopoli.improved.zeroth}, and then applying a Moser iteration argument, see Lemma \ref{lemma.final.Cacciopoli.recursion}. Finally, in Section \ref{sec.interiorregularity}, we prove Theorem \ref{teo.higherregularity} by applying the known elliptic regularity for the mixed operator $\mathcal{L} = -(-\Delta)^s + \Delta$ together with the results on the previous sections. This way, we determine the $C^{2,\alpha}_{loc}$ regularity in space of solutions to \eqref{1.1}.

\section{Preliminary notions and results} \label{sec.preliminary}

	In this section we will properly state definitions and results that will be used throughout the paper, taking special care with the regularity of our operator $\mathcal{L} = -(-\Delta)^s + \Delta$. Throughout this section, we will assume that $k \in \mathbb{N}$, $\alpha \in (0,1)$ and $s\in(0,1)$.
	
	First, let us define some of the spaces that will be used extensively in our work.
	\begin{align*}
		H_0^1(\Omega) =& \left\{\varphi \in L^2(\Omega) : \nabla \varphi \in L^2(\Omega) \; \text{ and } \; \text{Tr }\varphi = 0 \text{ on } \partial\Omega \right\} \\
		H^s(\mathbb{R}^n) =& \left\{ u\in L^2(\mathbb{R}^n) : [u]^2_{s} = \iint_{\mathbb{R}^n \times \mathbb{R}^n} \frac{(u(x)-u(y))^2}{|x-y|^{n+2s}} dx dy <\infty  \right\} \\
		H_0^s(\Omega) =& \{ u \in H^s(\mathbb{R}^n) : u=0 \text{ a.e. in } \mathbb{R}^n \setminus \Omega\}  \\
		D(\mathcal{L}) =& \left\{\varphi \in L^2(\Omega) : \mathcal{L} \varphi \in L^2(\Omega) \right\} \\
		C^k(\Omega) =& \left\{\varphi: \Omega \to \mathbb{R} : D^\beta \varphi \in C(\Omega) \text{ for every } |\beta| \leq k \right\}\\
		C^{k,\alpha} (\Omega) =& \left\{\varphi \in C^k(\Omega) : D^\beta \varphi \in C^{\alpha}(\Omega) \text{ for every } |\beta|=k \right\} \\
		C^{k,\alpha}_{loc} (\Omega) =& \left\{\varphi \in C^k(\Omega) : D^\beta \varphi \in C^{\alpha}_{loc}(\Omega) \text{ for every } |\beta|=k \right\} \\
		C^{k,1} (\Omega) =& \left\{\varphi \in C^k(\Omega) : D^\beta \varphi \in \text{Lip}(\Omega) \text{ for every } |\beta|=k \right\} \\
		C^{k,1}_{loc} (\Omega) =& \left\{\varphi \in C^k(\Omega) : D^\beta \varphi \in \text{Lip}_{loc}(\Omega) \text{ for every } |\beta|=k \right\}\\
		L^1_{2s} (\mathbb{R}^n) =& \left\{ \varphi \in L^1_{loc}(\mathbb{R}^n) : \int_{\mathbb{R}^n} \frac{|\varphi(y)|}{1+|y|^{n+2s}} dy <\infty \right\}.
	\end{align*}
	Notice that the space $H_0^1(\Omega)$ can be naturally identified with the space
	\begin{equation*}
		\tilde{H}_0^1(\Omega) = \left\{\varphi \in L^2(\mathbb{R}^n) : \nabla \varphi \in L^2(\mathbb{R}^n) \; \text{ and } \; u\equiv 0 \text{ a.e. in } \mathbb{R}^n\setminus \Omega  \right\},
	\end{equation*}
	which is the actual space we need to use due to the nonlocal nature of our operator $\mathcal{L}$, see \cite{BDVV2023}. However, we will identify these two spaces implicitly throughout the paper and every $H_0^1(\Omega)$ function with its extension on $\tilde{H}_0^1(\Omega)$.
	
	The space $L^1_{2s}(\mathbb{R}^n)$ is often called the ``tail space'', and it contains the functions whose ``decay at infinity'' is controlled enough for $(-\Delta)^s \varphi$ to be well defined. Indeed, it is well known that $(-\Delta)^s \varphi$ is well defined pointwise if $\varphi \in C^{1,1}_{loc}(\Omega) \cap L^1_{2s}(\mathbb{R}^n)$, see \cite{S2007}. In the same spirit, we will need to control this same ``decay at infinity'' throughout the article. For that reason, we define the nonlocal tail of a function $\varphi: \mathbb{R}^n \times (0,\tilde{T}) \rightarrow \mathbb{R}$ as follows:
	\begin{equation*}
		\text{Tail}_\infty (\varphi,x_0,r,t_1,t_2) = \sup_{t\in[t_1,t_2]} \left( r^{2s} \int_{\mathbb{R}^n \setminus B_r(x_0)} \frac{|\varphi(y,t)|}{|x_0-y|^{n+2s}} dy \right),
	\end{equation*}
	where $x_0 \in \mathbb{R}^n$ and $[t_1,t_2] \subset (0,\tilde{T})$. It is clear that this quantity is finite for every $\varphi \in L^\infty([t_1,t_2], L^1_{2s}(\mathbb{R}^n))$.
	
	Finally, we fix some of the notation we will be using. For some fixed $(x_0,t_0) \in \Omega \times (0,T)$ and $r>0$, we will often write the spatial balls and parabolic cylinders as follows:
	\begin{equation*}
		B_r := B_r(x_0) = \{ x\in \mathbb{R}^n: |x-x_0|<r\}, \;\; Q_r := B_r \times [t_0-r^2,t_0].
	\end{equation*}
	We will also use $\fint$ to denote the average integral, that is,
	\begin{equation*}
		\fint_K \varphi(z) dz = \frac{1}{|K|} \int_K \varphi(z) dz,
	\end{equation*}
	where $|K|$ denotes the Lebesgue measure of the set $K$.
	
	Finally, we note that $C$ and $c$ will be positive constants that may and will change value from line to line, and their dependence on certain data will be specified when it is necessary to do so.
	
	\section{Maximal semigroup regularity of the semilinear problem} \label{sec.existence}

	Let us first use the semigroup machinery to get existence and uniqueness of solutions to \eqref{1.1}, as well as some initial regularity. Although this process is standard, we include it here for completion.
	
	First, let us fix some $\varepsilon\in(0,\alpha)$ and consider a truncation of the reaction term, $f_\varepsilon: \mathbb{R} \to \mathbb{R}$ such that $f_\varepsilon(x) = f(x)$ for every $x\in (0,\alpha-\varepsilon)$, $f_\varepsilon$ is globally Lipschitz in $\mathbb{R}$ and it is increasing in $\mathbb{R}$. Now define the truncated problem:
    \begin{equation} \label{1.1.truncated}
    \begin{cases}
        \displaystyle (u_\varepsilon)_t(x,t)= \mathcal{L} u_\varepsilon(x,t)+f_\varepsilon(u_\varepsilon(x,t)), & (x,t)\in\Omega\times(0,T^\prime), \\
        u_\varepsilon(x,0)=u_0(x), & x\in\Omega,\\
        u_\varepsilon(x,t) = 0, & (x,t)\in\mathbb{R}^n \setminus \Omega \times [0,T^\prime),
    \end{cases}
    \end{equation}
    where $T^\prime \in (0,\infty]$ is the maximal time of existence of this new problem.
    
    Let us prove that our truncated system \eqref{1.1.truncated} admits a unique classical solution.
    
    \begin{lemma} \label{lemma.truncated.solution}
    	For every $u_0\in H_0^1(\Omega)$, system \eqref{1.1.truncated} has a unique classical solution $u_\varepsilon \in C((0,\infty), D(\mathcal{L})) \cap C([0,\infty),H_0^1(\Omega)) \cap C^1((0,\infty),L^2(\Omega))$.
    \end{lemma}
    
    \begin{proof}
    	It is a well known fact that $H_0^1(\Omega) \hookrightarrow H_0^s(\Omega)$ for every $0<s<1$, see \cite{DPV2012}. Then define the bilinear form
    	\begin{equation} \label{dirichlet.form}
    		\mathcal{E}(\varphi,\psi)= \frac{C_{n,s}}{2}\iint_{\mathbb{R}^n \times \mathbb{R}^n} \frac{(\varphi(x)-\varphi(y))(\psi(x)-\psi(y))}{|x-y|^{n+2s}} dy dx + \int_\Omega \nabla \varphi(x) \nabla \psi(x) dx.
    	\end{equation}
    	This form's natural domain is $H_0^1(\Omega) \cap H_0^s(\Omega) = H_0^1(\Omega)$. Furthermore, its associated norm is equivalent to the usual $H_0^1$ norm. Indeed,
    	\begin{equation*}
    		\|\nabla u\|_{L^2(\Omega)}^2 \leq \frac{C_{n,s}}{2} [u]^2_s + \|\nabla u\|_{L^2(\Omega)}^2 \leq  C \|\nabla u\|_{L^2(\Omega)}^2,
    	\end{equation*}
    	so it is clear that $\|u\|_{\mathcal{E}}^2  = \|u\|_{L^2(\Omega)}^2 + \mathcal{E}(u,u) \equiv \|u\|_{L^2(\Omega)}^2 + \|\nabla u\|_{L^2(\Omega)}^2 = \|u\|_{H_0^1(\Omega)}$. Therefore, $\mathcal{E}$ is a densely defined, symmetric, closed and nonnegative bilinear form. By Kato's representation theorem, see \cite{K1995}, we know that there exists a unique nonnegative self-adjoint operator $A$ associated with $\mathcal{E}$ that generates an analytical semigroup in $L^2(\Omega)$. This means that for every $u\in D(A)$ and $v\in H_0^1(\Omega)$,
    	\begin{equation*}
    		\mathcal{E}(u,v) = \langle Au,v \rangle_{L^2(\Omega)}.
    	\end{equation*}
    	According to how we defined $\mathcal{E}$, using integration by parts, it is easy to identify
    	\begin{equation*}
    		A = -\mathcal{L} = (-\Delta)^s - \Delta,
    	\end{equation*}
    	so we conclude that $-\mathcal{L}$ generates an analytical semigroup $e^{\mathcal{L} t}$ in $L^2(\Omega)$.
    	
    	Since $f_\varepsilon$ is globally Lipschitz, the Nemytskii operator $F_\varepsilon: L^2(\Omega) \to L^2(\Omega)$ is also globally Lipschitz. Bearing in mind that $u_0 \in H_0^1(\Omega) = D(\mathcal{E})$, we can use classical semigroup theory (see \cite[Proposition 7.1.10]{L1995}) to conclude that there exists a unique classical solution satisfying
    	\begin{equation*}
    		u_\varepsilon \in C((0,\infty), D(\mathcal{L})) \cap C([0,\infty),H_0^1(\Omega)) \cap C^1((0,\infty),L^2(\Omega)). \hfill \square
    	\end{equation*}
    \end{proof}
    
    Now we want to prove that solutions of the truncated problem \eqref{1.1.truncated} with initial data $u_0\in H_0^1(\Omega) \cap L^\infty(\Omega)$ such that $u_0(x) \in [0,\alpha-2\varepsilon]$ a.e. $x\in\Omega$ remain in $(0,\alpha-\varepsilon)$ up to some time $T_\varepsilon$. This way, we could identify the solutions of \eqref{1.1.truncated} with those of \eqref{1.1} and inherit the existence, uniqueness and regularity.
    
    To do so, we first need to establish a comparison result with respect to the initial data for equation \eqref{1.1.truncated}.
    
    \begin{lemma} \label{initial.data.comparison}
    	Let $u_0,u_1 \in H_0^1(\Omega)$ be such that $u_0 \geq u_1$ a.e. $x\in\Omega$. If $u_\varepsilon^0$ and $u_\varepsilon^1$ are the solutions of \eqref{1.1.truncated} with initial data $u_0$ and $u_1$ respectively, then
    	\begin{equation} \label{comparison.abstract}
    		u_\varepsilon^0(\cdot,t) \geq u^1_\varepsilon(\cdot,t) \text{ a.e. } x \in \Omega \text{ for all } t \in [0,\infty).
		\end{equation}
    \end{lemma}
    
    \begin{proof}
    	Due to classical semigroup theory, we know that the solution $u^i$ of system \eqref{1.1.truncated} with initial data $u_i$ is the unique fixed point of
    	\begin{equation*}
    		\mathcal{G}_i(w)(x,t) = e^{\mathcal{L} t} u_i(x) + \int_0^t e^{\mathcal{L}(t-s)} f_\varepsilon (w(x,s)) ds.
    	\end{equation*}
    	in $C([0,\tau], H_0^1(\Omega))$, since $\mathcal{G}_i$ is a contraction provided $\tau$ is small enough for $i=0,1$. Now we consider the sequence of Picard iterations,
    	\begin{equation*}
    		u^i_{k+1}(x,t) = \mathcal{G}_i(u^i_k)(x,t),
    	\end{equation*}
    	for every $(x,t) \in \Omega \times [0,\tau]$ and $k>1$, and $u^i_1 = u_i$. It is well known that this sequence $u^i_k$ converges to $u^i_\varepsilon$ in $C([0,\tau],H_0^1(\Omega))$.
    	
    	We know that $u^0_1 \geq u^1_1$ a.e. $x\in\Omega$. Assume now that $u^0_j \geq u^1_j$ a.e. $x\in\Omega$ for every $t\in[0,\tau]$ and every $j\leq k$, and let us prove that the ordering is still true for $j=k+1$. Indeed, we have
    	\begin{equation*}
    		u^i_{k+1}(x,t) = \mathcal{G}_i(u^i_k)(x,t) = e^{\mathcal{L}t} u_i(x) + \int_{0}^t e^{\mathcal{L}(t-s)} f_\varepsilon (u^i_k(x,s)) ds,
    	\end{equation*}
    	a.e. $x\in\Omega$ for every $t\in[0,\tau]$ and, since $f_\varepsilon$ is increasing, it is clear that $f_\varepsilon (u^0_k(x,s)) \geq f_\varepsilon (u^1_k(x,s))$ a.e. $x\in\Omega$ for every $t\in[0,\tau]$. Then, using that the semigroup is positivity preserving:
    	\begin{align*}
    		e^{\mathcal{L}t} u_0(x) &\geq e^{\mathcal{L}t} u_1(x), \\
    		e^{\mathcal{L}(t-s)} f_\varepsilon(u^0_k(x,s)) &\geq e^{\mathcal{L}(t-s)} f_\varepsilon(u^1_k(x,s)),
    	\end{align*}
    	a.e. $x\in\Omega$ for every $t\in[0,\tau]$. This means that $u^0_{k+1} \geq u^1_{k+1}$ a.e. $x\in\Omega$ for every $t\in[0,\tau]$. By induction, $u^0_n \geq u^1_n$ a.e. $x\in\Omega$ for every $t\in[0,\tau]$ and for every $n\in\mathbb{N}$. Using the convergence of this sequence to the respective solutions, we obtain that
    	\begin{equation*}
    		u^0_\varepsilon(x,t) \geq u^1_\varepsilon(x,t)
    	\end{equation*}
    	a.e. $x\in\Omega$ for every $t\in[0,\tau]$. Repeating this argument starting at time $\tau$, we can get the order preservation for every $t\in[0,\infty)$ and the result is proved. \qed
    \end{proof}
    
    Next, let us prove a comparison result for supersolutions and subsolutions of \eqref{1.1.truncated}. The ideas of the proof are very similar to those of Lemma \ref{initial.data.comparison}, but we include them anyway for completion.
    
    \begin{definition}
    	We say $\overline{u}\in C((0, \infty),D(\mathcal{L})) \cap C^1((0,\infty),L^2(\Omega))\cap C([0,\infty),L^1_{2s}(\mathbb{R}^n))$ is a supersolution of system \eqref{1.1.truncated} in $[t_0,t_1]$ if
    	 \begin{equation} \label{def.supersolution}
    \begin{cases}
        \displaystyle \overline{u}_t(x,t)\geq \mathcal{L}\overline{u}(x,t)+f_\varepsilon(\overline{u}(x,t)), & (x,t)\in\Omega\times [t_0,t_1], \\
        \overline{u}(x,0)\geq u_0(x), & x\in\Omega,\\
        \overline{u}(x,t) \geq 0, & (x,t)\in\mathbb{R}^n \setminus \Omega \times [t_0,t_1],
    \end{cases}
    \end{equation}
    a.e. $x\in\Omega$ for every $t\in[t_0,t_1]$. We say $\underline{u}\in C((0, \infty),D(\mathcal{L})) \cap C^1((0,\infty),L^2(\Omega))\cap C((0,\infty),L^1_{2s}(\mathbb{R}^n))$ is a subsolution of system \eqref{1.1.truncated} in $[t_0,t_1]$ if it satisfies the reverse inequalities a.e. $x\in\Omega$ for every $t\in[t_0,t_1]$.
    \end{definition}
    
    \begin{lemma} \label{lemma.comparison.abstract}
    	Let $u_\varepsilon$ be the solution to \eqref{1.1.truncated} with initial data $u_0$ and $\overline{u}$ a supersolution of \eqref{1.1.truncated} in $[0,\tau]$. Then
    	\begin{equation*}
    		\overline{u}(x,t) \geq u_\varepsilon(x,t)
    	\end{equation*}
    	a.e. $x\in\Omega$ and for every $t\in[0,\tau]$. The same is true for subsolutions with a reversed inequality.
    \end{lemma}
    \begin{proof}
    	First, we define
    	\begin{equation*}
    \overline{u}^\prime = \left\{
    \begin{array}{ll}
        \overline{u}, & x\in\Omega, \\
        0, & x\in\mathbb{R}^n\setminus \Omega.
    \end{array}
    \right.
    \end{equation*}
    and notice that
    \begin{equation*}
    	-(-\Delta)^s \overline{u}^\prime(x,t) + \Delta \overline{u}^\prime(x,t) \leq -(-\Delta)^s \overline{u}(x,t) + \Delta \overline{u}(x,t),
    \end{equation*}
    a.e. $x\in\Omega$ for every $t\in[0,\tau]$, which means that
    \begin{equation*}
    	\overline{u}^\prime_t(x,t) = \overline{u}_t(x,t) \geq \mathcal{L} \overline{u}(x,t) + f_\varepsilon(\overline{u}(x,t)) \geq \mathcal{L} \overline{u}^\prime(x,t) + f_\varepsilon(\overline{u}^\prime(x,t)),
    \end{equation*}
    a.e. $x\in\Omega$ for every $t\in[0,\tau]$. This means that there exists some non-negative $g:\mathbb{R} \to L^2(\Omega)$ such that
    	\begin{equation*}
    		\overline{u}^\prime_t(x,t) = \mathcal{L}\overline{u}^\prime(x,t)+f_\varepsilon(\overline{u}^\prime(x,t))+ g(t)(x)
    	\end{equation*}
    	a.e. $x\in\Omega$ for every $t\in[0,\tau]$. Since $\overline{u}^\prime$ satisfies the equation with the homogeneous Dirichlet boundary conditions, we know that it satisfies the `variation of constants' formula:
    	\begin{equation*}
    	\overline{u}^\prime(x,t) = e^{\mathcal{L}t}\overline{u}^\prime(x,0) + \int_{0}^t  e^{\mathcal{L}(t-s)}(f_\varepsilon(\overline{u}^\prime(x,s))+ g(s)(x)) ds,
    	\end{equation*}
    	and thanks to the positivity preserving semigroup:
    	\begin{equation} \label{eq.ordering}
    	\overline{u}^\prime(x,t) \geq e^{\mathcal{L}t} u_0(x) + \int_{0}^t  e^{\mathcal{L}(t-s)}(f_\varepsilon(\overline{u}^\prime(x,s))) ds = \mathcal{G}(\overline{u}^\prime)(x,t),
    	\end{equation}
    	a.e. $x\in\Omega$ for every $t\in[0,\tau]$. We also know that $u_\varepsilon$ is the unique fixed point of $\mathcal{G}$ in some interval $[0,\overline{\tau}]$. Choose $\tau_0 = \min\{\overline{\tau},\tau\}$, so everything is well defined up to time $\tau_0$.
    	
    	We again use the mechanism of Picard iterations in $C([0,\tau_0],H_0^1(\Omega))$. Define $u_{k+1}(x,t) = \mathcal{G}(u_k)(x,t)$ with $u_1(x,t) = \overline{u}^\prime(x,t)$. Then the sequence converges to $u_\varepsilon$ in $C([0,\tau_0],H_0^1(\Omega))$. Arguing by induction, we know that $\overline{u}^\prime \geq u_1$, and assume that $\overline{u}^\prime \geq u_j$ for every $j\leq k$. Let us prove that $\overline{u}^\prime \geq u_{k+1}$.
    	
    	We know thanks to the proof of Lemma \ref{initial.data.comparison} that $\mathcal{G}$ is increasing in $C([0,\tau_0],H_0^1(\Omega))$. Then using \eqref{eq.ordering} it is clear that
    	\begin{equation*}
    	\overline{u}^\prime(x,t) \geq \mathcal{G}(\overline{u}^\prime)(x,t) \geq \mathcal{G} (u_k)(x,t) = u_{k+1}(x,t),
    	\end{equation*}
    	a.e. $x\in\Omega$ for every $t\in[0,\tau_0]$. We have then proved that $\overline{u}^\prime \geq u_n$ for every $n\in\mathbb{N}$ and using the convergence of the sequence, we conclude that
    	\begin{equation*}
    		\overline{u}^\prime(x,t) \geq u_\varepsilon (x,t)
    	\end{equation*}
    	a.e. $x\in\Omega$ for every $t\in[0,\tau_0]$. Repeating this argument starting at time $\tau_0$, a typical continuation argument gives us that it is true for every $t\in[0,\tau]$. \qed
    \end{proof}
    
    \begin{corollary} \label{cor.positivity}
    	Let $u_0\in H_0^1(\Omega)$ such that $u_0(x)\geq 0$ a.e. $x\in\Omega$. Then the solution $u_\varepsilon$ of \eqref{1.1.truncated} with initial data $u_0$ satisfies $u_\varepsilon (x,t) \geq 0$ a.e. $x\in\Omega$ for every $t\in[0,\infty)$. 
    \end{corollary}
    
    \begin{proof}
    	The constant $0$ is a subsolution of system \eqref{1.1.truncated}. \qed
    \end{proof}
    
    Now that we have the comparison results on Lemma \ref{initial.data.comparison} and Lemma \ref{lemma.comparison.abstract}, we are able to prove that the original system \eqref{1.1} also has a classical solution.
    
    \begin{proof}[Proof of Proposition \ref{prop.semigroup.regularity}]
		First, assume that $\alpha<\infty$ and fix some $0<\varepsilon<(\alpha-\|u_0\|_\infty)/2$. Then consider the truncation $f_\varepsilon$ of $f$ in $(0,\alpha-\varepsilon)$.	Let $u_\varepsilon$ be the unique classical solution obtained for the truncated system \eqref{1.1.truncated} with initial data $u_0$ in Lemma \ref{lemma.truncated.solution}. Next, introduce the auxiliary ODE problem
    	 \begin{equation} \label{aux.supersolution}
    \left\{
    \begin{array}{ll}
        w_t(t)= f_\varepsilon(w(t)), &\\
        w(0)= \|u_0\|_\infty. &
    \end{array}
    \right.
    \end{equation}
    Since $f_\varepsilon$ is globally Lipschitz, Peano's and Picard-Lindelöf's theorems tell us that there exists a unique continuous local solution to \eqref{aux.supersolution}. What is more, since $w(0) = \|u_0\|_\infty \leq \alpha-2\varepsilon$, we can ensure by continuity that there exists some time $T_\varepsilon$ (that only depends on $\varepsilon$) such that $w(t) \leq \alpha-\varepsilon$ for every $t\in[0,T_\varepsilon]$. Let us now prove that $w(x,t) = w(t)$ for every $(x,t)\in\mathbb{R}^n \times [0,T_\varepsilon]$ is a supersolution of \eqref{1.1.truncated}.
    
    First, it is clear that $w(x,t) \geq 0$ for every $(x,t) \in \mathbb{R}^n \setminus\Omega \times [0,T_\varepsilon]$. It is also immediate to notice that $w(x,0) \geq u_0(x)$ for every $x\in\Omega$. Finally, notice that, since $w$ is constant in space:
    \begin{equation*}
    	\mathcal{L} w(x,t) = -(-\Delta)^s w(x,t) + \Delta w(x,t) = 0,
    \end{equation*}
    so $w_t(x,t) = \mathcal{L} w(x,t) + f_\varepsilon(w(x,t))$ for every $(x,t) \in \Omega \times[0,T_\varepsilon]$. Then $w$ is a supersolution of \eqref{1.1.truncated} and by Lemma \ref{lemma.comparison.abstract} and Corollary \ref{cor.positivity} we know that
    \begin{equation*}
    	\alpha-\varepsilon \geq w(x,t) \geq u_\varepsilon(x,t) \geq 0,
    \end{equation*}
    a.e. $x\in\Omega$ for every $t\in[0,T_\varepsilon]$.
    
    Then, thanks to the definition of $f_\varepsilon$, we have that $f_\varepsilon(u_\varepsilon(\cdot,t)) = f(u_\varepsilon(\cdot,t))$ for every $t\in[0,T_\varepsilon]$. Therefore we conclude that $u_\varepsilon$ is actually a classical solution of \eqref{1.1} for every $t\in[0,T_\varepsilon]$. 
    
    Doing this process again starting at time $T_\varepsilon$, maybe choosing some smaller $\varepsilon^\prime$, we could extend our solution to some time interval $[0,T_\varepsilon+T_{\varepsilon^\prime}]$. We could then repeat this as long as our solution satisfies $\|u(\cdot,t)\|_\infty < \alpha$ and conclude that there exists some time $T>0$ such that \eqref{1.1} has a classical solution
    \begin{equation*}
    	u \in C((0,T), D(\mathcal{L})) \cap C([0,T),H_0^1(\Omega)) \cap C^1((0,T),L^2(\Omega)),
    \end{equation*}
    and that $\limsup_{t\to t_0} \|u(\cdot,t)\|_\infty = \alpha$ if and only if $t_0=T$.

	If $\alpha = +\infty$, then we consider the truncation $f_m$ of $f$ in $(0, m)$ where $m = \max\{ 1, 2\|u_0\|_{L^\infty(\Omega)} \}$ and following the same reasoning we reach the result. \qed
    \end{proof}
    
    \section{Hölder continuity via the Campanato-Meyers characterization} \label{sec.campanato}
    
	We will now prove further regularity of the solutions of \eqref{1.1} using a Campanato style argument. Our objective is to find, for any $(x_0,t_0)\in \Omega \times (0,T)$, some $r^\prime>0$ and $\alpha>0$ such that
	\begin{equation*}
		\fint_{t_0-r^2}^{t_0} \fint_{B_r(x_0)} |u(x,t)-(u)_{B_r(x_0) \times (t_0-r^2,t_0)}|^2  dx dt \leq C r^{2\alpha},
	\end{equation*}
	for every $r\leq r^\prime$. Via the Campanato integral characterization, see \cite{G2003}, this gives us that $u\in C^\alpha_{loc}(\Omega \times (0,T))$, which is our objective.
	
	This section is heavily based on \cite{D2024}, with some changes to accommodate the effect of the reaction term on all the estimates. First we will establish a Caccioppoli inequality for equation \eqref{1.1}, then state a difference estimate comparing against the classical heat equation and, finally, use a Campanato decay argument to prove Hölder regularity. We include all the details for completeness and to avoid omitting key steps.

\subsection{Caccioppoli estimate}

In this subsection we will establish the Caccioppoli inequality for our equation. The first step should be the regularization procedure with respect to the time variable, which is a straightforward process that has been tackled multiple times such as in \cite{DZZ2021}. We will skip this step here.

\begin{proposition} \label{prop.cacciopoli}
	Let $u$ be the unique solution of problem \eqref{1.1}. Let $x_0\in\Omega$, $r>0$ and $B_r = B_r(x_0)$ such that $\overline{B}_r \subset \Omega$. Let also $\tau_2>\tau_1>0$ such that they satisfy $[\tau_1,\tau_2] \subset (0,T)$. Then for any nonnegative functions $\psi \in C_0^\infty (B_r)$ and $\eta\in C^\infty(\mathbb{R})$ such that $\eta(t) \equiv 0$ if $t\leq \tau_1$, it holds that
	
	\begin{align} \label{eq.caccio.ineq}
		\nonumber & \sup_{t\in(\tau_1,\tau_2)} \left( \int_{B_r} v^2 \eta^2 \psi^2 (x,t) dx \right) +\int_{\tau_1}^{\tau_2} \int_{B_r} |\nabla v|^2 \psi^2 \eta^2 (x,t) dx dt\\
		\nonumber & + \int_{\tau_1}^{\tau_2} \int_{B_r} \int_{B_r} \frac{|v\psi(x,t) - v\psi(y,t)|^2 \eta^2(t)}{|x-y|^{n+2s}} dxdydt  \\
		 \leq& \; C \int_{\tau_1}^{\tau_2} \int_{B_r} v^2 \eta |\eta_t| \psi^2 (x,t) dx dt + \int_{\tau_1}^{\tau_2} \int_{B_r} |\nabla \psi|^2 v^2 \eta^2 dx dt\\
		 \nonumber &+  C \int_{\tau_1}^{\tau_2} \int_{B_r} \int_{B_r} \frac{\max\{|v(x,t)|,|v(y,t)|\}^2}{|x-y|^{n+2s}}  |\psi(x)-\psi(y)|^2 \eta^2(t) dx dy dt \\
		 \nonumber &+ C \int_{\tau_1}^{\tau_2} \int_{B_r} \left( \int_{\mathbb{R}^n \setminus B_r} \frac{|v(y,t)|}{|x-y|^{n+2s}} dy \right) |v(x,t)| \psi^2(x) \eta^2(t) dx dt \\
		 \nonumber&+ C \int_{\tau_1}^{\tau_2} \int_{B_r} |f(u(x,t))|^2 dx dt + C \int_{\tau_1}^{\tau_2} \int_{B_r} |v \psi^2 \eta^2 (x,t)|^2 dx dt,
	\end{align}
	where $C=C(n,s)$ and $v(x,t) = u(x,t)-k$ for some $k\in\mathbb{R}$.
\end{proposition}

\begin{proof}
	First, define $v(x,t) = (u-k)(x,t)$ for every $(x,t)\in\Omega\times[0,T)$ with some $k\in\mathbb{R}$. Next, choose $\varphi(x,t) = v(x,t) \psi^2(x) \eta^2(t)$ as a test function for equation \eqref{1.1}, which after integrating over $[t_1,t_2] \subset (\tau_1,\tau_2)$, yields
	\begin{equation}
	\label{eq.caccio.1}
	\begin{aligned} 
		& \int_{t_1}^{t_2} \int_{B_r}  u_t(x,t) (v \psi^2 \eta^2)(x,t) dx dt + \int_{t_1}^{t_2} \mathcal{B}(u, v\psi^2\eta^2) dt \\
		+& \; \int_{t_1}^{t_2} \int_{B_r}\hspace{-0.25cm} \nabla u \nabla (v \psi^2 \eta^2) (x,t) dx dt - \int_{t_1}^{t_2} \int_{B_r} f(u(x,t)) (v\psi^2\eta^2) (x,t) dx dt = 0,
	\end{aligned}
\end{equation}
	where
	\begin{equation*}
	\mathcal{B}(u,v) = \frac{C_{n,s}}{2} \iint_{\mathbb{R}^n \times \mathbb{R}^n} \frac{(u(x,t)-u(y,t))(v(x,t)-v(y,t))}{|x-y|^{n+2s}} dx dy.
\end{equation*}
	Let us treat the terms in \eqref{eq.caccio.1} one by one. First, by noticing that $u_t=v_t$, the term with the time derivative can be expanded into
	\begin{align*}
	\mathcal{I}_1 =& \; \int_{t_1}^{t_2} \int_{B_r} u_t(x,t) (v \psi^2 \eta^2)(x,t) dx dt \\
	&= \frac{1}{2}\int_{t_1}^{t_2} \int_{B_r} (v^2 \eta^2)_t \psi^2 (x,t) dx dt - \int_{t_1}^{t_2} \int_{B_r} v^2 \eta \eta_t \psi^2 (x,t) dx dt \\
	&= \frac{1}{2} \left(\int_{B_r} v^2 \eta^2 \psi^2 (x,t_2) dx - \int_{B_r} v^2 \eta^2 \psi^2 (x,t_1) dx  \right)  - \int_{t_1}^{t_2} \int_{B_r} v^2 \eta \eta_t \psi^2 (x,t) dx dt,
	\end{align*}
	and since we know that $\eta(t) \equiv 0$ for every $t\leq \tau_1$, then it is clear that
	\begin{equation} \label{eq.caccio.2}
	\mathcal{I}_1 =  \frac{1}{2} \int_{B_r} v^2 \eta^2 \psi^2 (x,t_2) dx - \int_{t_1}^{t_2} \int_{B_r} v^2 \eta \eta_t \psi^2 (x,t) dx dt.
	\end{equation}
	Next, to estimate the nonlocal integral we proceed as usual: the integral over $dy$ is split in two parts: one over $\Omega$ and another one over $\mathbb{R}^n \setminus \Omega$. Then both integrals are estimated separately. The process is long and tedious, but fairly standard, so we refer to \cite{DZZ2021} for more details, and state that
	\begin{align} \label{eq.caccio.3}
		\nonumber \mathcal{I}_2 =& \; \int_{t_1}^{t_2} \mathcal{B}(u, v\psi^2\eta^2) dt  \\
		\geq& \; C \int_{t_1}^{t_2} \int_{B_r} \int_{B_r} \frac{|v\psi(x,t) - v\psi(y,t)|^2 \eta^2(t)}{|x-y|^{n+2s}} dydxdt \\
		\nonumber &-  C \int_{t_1}^{t_2} \int_{B_r} \int_{B_r} \frac{\max\{|v(x,t)|,|v(y,t)|\}^2}{|x-y|^{n+2s}}  |\psi(x)-\psi(y)|^2 \eta^2(t) dx dy dt \\
		\nonumber &-  C \int_{t_1}^{t_2} \int_{B_r} \left( \int_{\mathbb{R}^n \setminus B_r} \frac{|v(y,t)|}{|x-y|^{n+2s}} dy \right) |v(x,t)| \psi^2(x) \eta^2(t) dx dt.
	\end{align}
	Then, bearing in mind that $\nabla u = \nabla v$, let us expand the local part of the operator as usual:
	\begin{align}\label{eq.caccio.4}
		\nonumber \mathcal{I}_3 =& \; \int_{t_1}^{t_2} \int_{B_r} \nabla u \nabla (v \psi^2 \eta^2) dx dt \\
		 =& \; \int_{t_1}^{t_2} \int_{B_r} |\nabla v|^2 \psi^2 \eta^2 dx dt + 2 \int_{t_1}^{t_2} \int_{B_r} v \nabla v \psi \nabla \psi \eta^2 dx dt,
	\end{align}
	and notice that, by Young's inequality,
	\begin{align} \label{eq.caccio.5}
	\nonumber - 2 \int_{t_1}^{t_2} \int_{B_r} v \nabla v \psi \nabla \psi \eta^2 dx dt &\leq \, 2 \int_{t_1}^{t_2} \int_{B_r} |v \nabla \psi| |\psi \nabla v| \eta^2 dx dt \\
	&\leq \, \frac{1}{2} \int_{t_1}^{t_2} \int_{B_r} |\nabla v|^2 \psi^2 \eta^2 dx dt + 2 \int_{t_1}^{t_2} \int_{B_r} |\nabla \psi|^2 v^2 \eta^2 dx dt.
	\end{align}
	For the last integral, we use Hölder's and Young's inequalities to find that
	\begin{align} \label{eq.caccio.6}
		\nonumber \mathcal{I}_4 =& \; \int_{t_1}^{t_2} \int_{B_r} f(u(x,t)) (v\psi^2\eta^2) (x,t) dx dt \\
		&\leq \left(\int_{t_1}^{t_2} \int_{B_r} |f(u(x,t))|^2 dx dt \right)^{\frac{1}{2}} \left(\int_{t_1}^{t_2} \int_{B_r} |v \psi^2 \eta^2 (x,t)|^2 dx dt \right)^{\frac{1}{2}} \\
		\nonumber &\leq \frac{1}{2} \int_{t_1}^{t_2} \int_{B_r} |f(u(x,t))|^2 dx dt + \frac{1}{2} \int_{t_1}^{t_2} \int_{B_r} |v \psi^2 \eta^2 (x,t)|^2 dx dt
	\end{align}
	Finally, using \eqref{eq.caccio.2}, \eqref{eq.caccio.3}, \eqref{eq.caccio.4}, \eqref{eq.caccio.5} and \eqref{eq.caccio.6} and remembering that \eqref{eq.caccio.1} gave us
	\begin{equation*}
	 	\mathcal{I}_1 + \mathcal{I}_2 + \mathcal{I}_3 - \mathcal{I}_4 = 0,
	\end{equation*}	
	we get
	\begin{align*}
		\int_{B_r} & v^2 \eta^2 \psi^2 (x,t_2) dx + \int_{t_1}^{t_2} \int_{B_r} \int_{B_r} \frac{|v\psi(x,t) - v\psi(y,t)|^2 \eta^2(t)}{|x-y|^{n+2s}} dydxdt +\int_{t_1}^{t_2} \int_{B_r} |\nabla v|^2 \psi^2 \eta^2 (x,t) dx dt \\
		 \leq& \; C \int_{\tau_1}^{\tau_2} \int_{B_r} v^2 \eta \eta_t \psi^2 (x,t) dx dt +  C \int_{\tau_1}^{\tau_2} \int_{B_r} \int_{B_r} \frac{\max\{|v(x,t)|,|v(y,t)|\}^2}{|x-y|^{n+2s}}  |\psi(x)-\psi(y)|^2 \eta^2(t) dx dy dt \\
		 &+ C \int_{\tau_1}^{\tau_2} \int_{B_r} \left( \int_{\mathbb{R}^n \setminus B_r} \frac{|v(y,t)|}{|x-y|^{n+2s}} dy \right) |v(x,t)| \psi^2(x) \eta^2(t) dx dt + C \int_{\tau_1}^{\tau_2} \int_{B_r} |\nabla \psi|^2 v^2 \eta^2 dx dt \\
		 & + C \int_{\tau_1}^{\tau_2} \int_{B_r} |f(u(x,t))|^2 dx dt + C \int_{\tau_1}^{\tau_2} \int_{B_r} |v \psi^2 \eta^2 (x,t)|^2 dx dt,
	\end{align*}
	where we have used that all the terms in the right hand side are positive to expand their limits of integration to the whole $(\tau_1,\tau_2)$. Then, since the right hand side is independent of $t_1$ and $t_2$, we can take the supremum on the left hand side over $(\tau_1,\tau_2)$ as well to get the result. \qed
\end{proof}

Next, we finish our Caccioppoli inequality via an appropriate choice of $\psi$ and $\eta$.

\begin{corollary}
	Let $u$ be the unique solution of problem \eqref{1.1}. Let $(x_0,t_0) \in \Omega \times(0,T)$ and $r>0$ such that $\overline{B}_r \subset \Omega$ and $[t_0-r^2,t_0] \subset [0,T)$. Then it holds that
	\begin{align} \label{eq.caccio.final}
		\nonumber \sup_{t\in(t_0-r^2/4,t_0)}& \left( \frac{1}{r^2}\fint_{B_r} |u(x,t)-(u)_{Q_r}|^2\psi^2(x,t) dx \right) + \fint_{t_0-r^2/4}^{t_0} \fint_{B_{r/2}} |\nabla u(x,t)|^2 \\
		\nonumber& + \fint_{t_0-r^2/4}^{t_0} \int_{B_{r/2}} \fint_{B_{r/2}} \frac{|u(x,t)-u(y,t)|^2}{|x-y|^{n+2s}} dx dy dt \\
		\leq & \; \left(\frac{C}{r^2} + \frac{C}{r^{2s}} + C \right) \fint_{t_0-r^2}^{t_0} \fint_{B_r} |u(x,t)-(u)_{Q_r}|^2(x,t) dx dt  \\
		\nonumber& + C \int_{t_0-r^2}^{t_0} \left(\int_{\mathbb{R}^n \setminus B_r} \frac{|u(y,t)-(u)_{Q_r}|}{|x_0-y|^{n+2s}}dy \right)^2 dt \\
		\nonumber& + C \fint_{t_0-r^2}^{t_0} \fint_{B_r} |f(u(x,t))|^2 dx dt,
	\end{align}
	where $C=C(n,s)$.
\end{corollary}

\begin{proof}
	Consider the following choices in the expression \eqref{eq.caccio.ineq}:  take $(\tau_1,\tau_2) = (t_0-r^2,t_0)$, take the spatial cutoff function $\psi$ such that $\psi \equiv 1$ in $B_{r/2}$, $\psi \equiv 0$ in $\mathbb{R}^n \setminus B_{3r/4}$, $\psi\leq 1$ in $B_r$ and $|\nabla \psi| \leq C/r$ in $B_r$, and the time cutoff function $\eta$ such that $\eta(t_0-r^2) = 0$, $\eta\equiv 1$ in $(t_0-r^2/4,t_0)$, $\eta \leq 1$ in $(t_0-r^2,t_0)$ and $|\eta_t| \leq C/r^2$ in $(t_0-r^2,t_0)$. Take as well $k = (u)_{Q_r}$. With this, and using that every single term is positive, we get
	\begin{align} \label{eq.cutoff.1}
		\nonumber& \sup_{t\in(t_0-r^2/4,t_0)} \left( \int_{B_r} v^2\psi^2(x,t) dx \right) + \int_{t_0-r^2/4}^{t_0} \int_{B_{r/2}} |\nabla v(x,t)|^2 \\
		\nonumber& + \int_{t_0-r^2/4}^{t_0} \int_{B_{r/2}} \int_{B_{r/2}} \frac{|v(x,t)-v(y,t)|^2}{|x-y|^{n+2s}} dx dy dt \\
		\leq & \; \frac{C}{r^2} \int_{t_0-r^2}^{t_0} \int_{B_r} v^2(x,t) dx dt + C \int_{t_0-r^2}^{t_0} \int_{B_r} v^2 (x,t) dx dt \\
		\nonumber & + C \int_{t_0-r^2}^{t_0} \int_{B_r} \int_{B_r} \frac{\max\{|v(x,t)|,|v(y,t)|\}^2}{|x-y|^{n+2s}} |\psi(x)-\psi(y)|^2 dx dy dt \\
		\nonumber& + C \int_{t_0-r^2}^{t_0} \int_{B_r} \left(\int_{\mathbb{R}^n \setminus B_r} \frac{|v(y,t)|}{|x-y|^{n+2s}}dy \right) |v(x,t)| dx dt + C \int_{t_0-r^2}^{t_0} \int_{B_r} |f(u(x,t))|^2 dx dt.
	\end{align}
	We then need to estimate the nonlocal terms and try to absorb what we can in our more practical terms. First, notice that, since $|\nabla \psi| \leq C/r$ in $B_r$, then
	\begin{equation*}
		\frac{|\psi(x)-\psi(y)|^2}{|x-y|^{n+2s}} \leq \frac{C}{r^2} \frac{1}{|x-y|^{n+2s-2}}.
	\end{equation*}
	With this, it is easy to see that the first nonlocal integral satisfies
	\begin{align} \label{eq.cutoff.2}
		\nonumber \int_{t_0-r^2}^{t_0} & \int_{B_r} \int_{B_r} \frac{\max\{|v(x,t)|,|v(y,t)|\}^2}{|x-y|^{n+2s}} |\psi(x)-\psi(y)|^2 dx dy dt \\
		\leq& \;\frac{C}{r^2} \int_{t_0-r^2}^{t_0} \int_{B_r} \int_{B_r} \frac{v^2(x,t)}{|x-y|^{n+2s-2}} dxdydt \\
		\nonumber \leq& \;\frac{C}{r^{2s}} \int_{t_0-r^2}^{t_0} \int_{B_r} v^2(x,t) dx dt  = \frac{C}{r^{2s}} \int_{t_0-r^2}^{t_0} \int_{B_r} |u(x,t)-(u)_{Q_r}|^2 dx dt.
	\end{align}	
	For the second integral, we recall that $\psi\equiv 0$ in $\mathbb{R}^n \setminus B_{3r/4}$. Then, if $y\in\mathbb{R}^n \setminus B_r$ and $x \in B_{3r/4}$:
	\begin{equation*}
		|x-y| \geq \frac{1}{4} |x_0-y|.
	\end{equation*}
	We can use this to estimate the second nonlocal integral. First, we take the average integral over $Q_r$ and then, with the help of Hölder's and Young's inequalities.
	\begin{align} \label{eq.cutoff.3}
		\nonumber \fint_{t_0-r^2}^{t_0} \fint_{B_r} & \left(\int_{\mathbb{R}^n \setminus B_r} \frac{|v(y,t)|}{|x-y|^{n+2s}}dy \right) |v(x,t)| dx dt \\
		\leq& \; C \left( \frac{1}{r^2} \int_{t_0-r^2}^{t_0} \left( \int_{\mathbb{R}^n \setminus B_r} \frac{|u(y,t)-(u)_{Q_r}|}{|x_0-y|^{n+2s}} dy\right)^2 dt \right)^{1/2} \left(\fint_{t_0-r^2}^{t_0} \fint_{B_r} |u(x,t)-(u)_{Q_r}|^2 dx dt \right)^{1/2} \\
		\nonumber \leq& \; C \int_{t_0-r^2}^{t_0} \left( \int_{\mathbb{R}^n \setminus B_r} \frac{|u(y,t)-(u)_{Q_r}|}{|x_0-y|^{n+2s}} dy\right)^2 dt + \frac{C}{r^2} \fint_{t_0-r^2}^{t_0} \fint_{B_r} |u(x,t)-(u)_{Q_r}|^2 dx dt.
	\end{align}
	Now, plugging \eqref{eq.cutoff.2} and \eqref{eq.cutoff.3} into \eqref{eq.cutoff.1}, we get
	\begin{align*}
		\nonumber \sup_{t\in(t_0-r^2/4,t_0)} & \left( \int_{B_r} |u(x,t)-(u)_{Q_r}|^2\psi^2(x,t) dx \right) + \int_{t_0-r^2/4}^{t_0} \int_{B_{r/2}} |\nabla u(x,t)|^2 \\
		\nonumber& + \int_{t_0-r^2/4}^{t_0} \int_{B_{r/2}} \int_{B_{r/2}} \frac{|u(x,t)-u(y,t)|^2}{|x-y|^{n+2s}} dx dy dt \\
		\leq & \; \left(\frac{C}{r^2} + \frac{C}{r^{2s}} + C \right) \int_{t_0-r^2}^{t_0} \int_{B_r} |u(x,t)-(u)_{Q_r}|^2(x,t) dx dt  \\
		\nonumber& + C r^{n+2} \int_{t_0-r^2}^{t_0} \left(\int_{\mathbb{R}^n \setminus B_r} \frac{|u(y,t)-(u)_{Q_r}|}{|x_0-y|^{n+2s}}dy \right)^2 dt + C \int_{t_0-r^2}^{t_0} \int_{B_r} |f(u(x,t))|^2 dx dt.
	\end{align*}
	Finally, dividing both sides by $r^{n+2}$ to get the average integrals on $x$ and $t$, we arrive at the desired result. \qed
\end{proof}

 \subsection{Difference estimate}
 Next, we want to establish an energy difference inequality with the solution of a problem for which we already know our desired regularity. We will then be able to use said inequality to transfer the regularity in some way. Define the auxiliary problem
    \begin{equation} \label{1.Das}
    \begin{cases}
        \displaystyle w_t(x,t)= \Delta w(x,t), & (x,t)\in B_{r/4} \times (t_0-r^2/16,t_0), \\
        w(x,t_0-r^2/16)=u(x,t_0-r^2/16), & x\in B_{r/4},\\
        w=u, & (x,t)\in\mathbb{R}^n \setminus B_{r/4} \times (t_0-r^2/16,t_0).
    \end{cases}
    \end{equation}
    We already know that $u\in C((0,T), H_0^1(\Omega))$, which implies that $u(\cdot,t_0-r^2/16) \in H^1(B_{r/4})$ and $u \in L^2 (\partial(B_{r/4}) \times (t_0-r^2/16,t_0)$. Then the initial data and boundary conditions are well prescribed, and it is clear that there exists a unique solution $w\in C((0,T),u+H_0^1(B_{r/4}))$ of problem \eqref{1.Das}, see \cite{E2010}.
    
    We state the Poincaré inequality and some regularity estimates that are well known for the solution to the heat equation. These stem from classical results found in any evolution PDE book such as \cite{L1996}.
    \begin{lemma} \label{lemma.heatproperties}
    	Let $w$ be the unique solution to system \eqref{1.Das}. Then, the following inequalities hold for every $d\in(0,r/8]$:
    	\begin{equation} \label{eq.poincare.w}
    		\fiint_{Q_d}|w-(w)_{Q_d}|^2 dx dt \leq C d^2 \fiint_{Q_d} |\nabla w|^2 dx dt,
    	\end{equation}
    	\begin{equation}
    		\fiint_{Q_d} w^2(x,t) dx dt \leq C \fiint_{Q_{r/8}} w^2(x,t) dx dt,
    	\end{equation}
    	\begin{equation}
    		\fiint_{Q_d} |w(x,t)-(w)_{Q_d}|^2 dx dt \leq C \left(\frac{d}{r}\right)^2 \fiint_{Q_{r/8}} |w(x,t)-(w)_{Q_{r/8}}|^2 dx dt,
    	\end{equation}
    	where $C=C(n)$.
    \end{lemma}
    
    Let us now find the energy estimate on $u-w$.
    
    \begin{proposition} \label{prop.difference}
    	Let $u$ be the unique solution of problem \eqref{1.1} and $w$ the unique solution of \eqref{1.Das}. Let $(x_0,t_0) \in \Omega \times (0,T)$ and $r>0$ such that $\overline{B}_r \subset \Omega$ and $[t_0-r^2,t_0] \subset [0,T)$. Then it holds that
    	\begin{align} \label{eq.difference.estimate}
    	\nonumber\frac{1}{2} & \int_{B_{r/4}} (u-w)^2(x,t_0) dx + \iint_{Q_{r/4}} |\nabla(u-w)(x,t)|^2 dx dt \\
    	\leq& \; C r^{2(1-s)} \int_{t_0-r^2/4}^{t_0} \int_{B_{r/2}} \int_{B_{r/2}} \frac{|u(x,t)-u(y,t)|^2}{|x-y|^{n+2s}} dx dy dt + C r^{n+2} \int_{t_0-r^2}^{t_0} \left( \int_{\mathbb{R}^n \setminus B_r} \frac{|u(y,t)-(u)_{Q_r}|}{|x_0-y|^{n+2s}} dy\right)^2 dt \\
    	\nonumber &+ C r^{2-4s}  \int_{t_0-r^2}^{t_0} \int_{B_{r}} |u(x,t) - (u)_{Q_r}|^2 dx dt +  Cr^2 \iint_{Q_{r/4}} |f(u(x,t))|^2 dx dt,
    	\end{align}
    	where $C=C(n,s)$.
    \end{proposition}
    
    \begin{proof}
    	The first step of this proof should again be the regularization procedure with respect to the time variable we mentioned earlier. We skip the step and refer to \cite{DZZ2021} again.
    	
    	The function $z = u-w$ is clearly in $C((t_0-r^2/16,t_0), H_0^1(B_{r/4}))$, and satisfies that $z\equiv 0$ on $\mathbb{R}^n \setminus B_{r/4}$. Then it is clear that we can use it as a test function on both \eqref{1.1} and \eqref{1.Das}. We do so, integrate in time over $(t_0-r^2/16,t_0)$ and subtract both expressions to get
    	\begin{align} \label{eq.difference.1}
    		\nonumber \frac{1}{2} & \iint_{Q_{r/4}} ((u-w)^2)_t(x,t) dx dt + \iint_{Q_{r/4}} |\nabla(u-w)(x,t)|^2 dx dt \\
    		 &+ \int_{t_0-r^2/16}^{t_0} \mathcal{B}(u, u-w) dt - \iint_{Q_{r/4}} f(u(x,t)) (u-w)(x,t) dx dt = 0.
    	\end{align}
    	Let us again study the terms one by one. First, the term with the time derivative turns directly into
    	\begin{equation} \label{eq.difference.2}
    	\frac{1}{2} \iint_{Q_{r/4}} ((u-w)^2)_t(x,t) dx dt = \frac{1}{2} \int_{B_{r/4}} (u-w)^2(x,t_0) dx.
    	\end{equation}
    	The nonlocal term is again a difficult beast to deal with, so we refer to the calculations done in \cite{D2024} and state that
    	\begin{align} \label{eq.difference.3}
    		\nonumber \Bigg| & \int_{t_0-r^2/16}^{t_0}  \mathcal{B}(u, u-w) dt \Bigg| \\
    		\leq& \; C r^{2(1-s)} \int_{t_0-r^2/4}^{t_0} \int_{B_{r/2}} \int_{B_{r/2}} \frac{|u(x,t)-u(y,t)|^2}{|x-y|^{n+2s}} dx dy dt + \frac{1}{4}\int_{t_0-r^2/16}^{t_0} \int_{B_{r/2}} |\nabla (u-w)(x,t)|^2 dx dt \\
    		\nonumber &+ C r^{2-4s}  \int_{t_0-r^2}^{t_0} \int_{B_{r}} |u(x,t) - (u)_{Q_r}|^2 dx dt + C r^{n+2} \int_{t_0-r^2}^{t_0} \left( \int_{\mathbb{R}^n \setminus B_r} \frac{|u(y,t)-(u)_{Q_r}|}{|x_0-y|^{n+2s}} dy\right)^2 dt.
    	\end{align}
    	Notice that in \cite{D2024} the calculations are done with the unnormalized integral fractional laplacian, while we are using the normalized one. This is not a problem, since the constant $C_{n,s}$ can be absorbed in the generic constants, and the factor $1/4$ can be obtained by choosing the parameter in Young's inequality appropriately. Finally, the reaction term can be estimated like follows by use of the Poincaré, Hölder's and Young's inequalities.
    	\begin{align} \label{eq.difference.4}
    	\nonumber \iint_{Q_{r/4}} & |f(u(x,t))||(u-w)(x,t)| dx dt \\
    	 \nonumber&\leq  \int_{t_0-r^2/16}^{t_0} \|f(u(\cdot,t))\|_{L^2(B_{r/4})} \|(u-w)(\cdot,t)\|_{L^2(B_{r/4})} dt \\
    	&\leq Cr \int_{t_0-r^2/16}^{t_0} \|f(u(\cdot,t))\|_{L^2(B_{r/4})} \|\nabla (u-w)(\cdot,t)\|_{L^2(B_{r/4})} dt \\
    	\nonumber &\leq \frac{1}{4} \iint_{Q_{r/4}} |\nabla(u-w)|^2 dx dt + Cr^2 \iint_{Q_{r/4}} |f(u(x,t))|^2 dx dt.
    \end{align}
    	Then, plugging \eqref{eq.difference.2}, \eqref{eq.difference.3} and \eqref{eq.difference.4} into \eqref{eq.difference.1}, we get the desired result.  \qed   	
    \end{proof}
    
    Finally, we merge both crucial estimates we have found to get rid of the purely nonlocal term on \eqref{eq.difference.estimate}.
    
    \begin{corollary}
    	Let $u$ be the unique solution of problem \eqref{1.1} and $w$ the unique solution of \eqref{1.Das}. Let $(x_0,t_0) \in \Omega \times (0,T)$ and $r\in(0,1]$ such that $\overline{B}_r \subset \Omega$ and $[t_0-r^2,t_0] \subset [0,T)$. Then, for every $\varepsilon<1-s$, it holds that
    	\begin{align} \label{eq.merged.estimate}
   \nonumber  &\fiint_{Q_{r/4}} |(u-w)(x,t)|^2 dx dt \leq \; Cr^2  \fiint_{Q_{r/4}} |\nabla(u-w)(x,t)|^2 dx dt \\
    	\leq&\; C r^{\varepsilon} \left[ r^{2-\varepsilon} \int_{t_0-r^2}^{t_0}\left(\int_{\mathbb{R}^n \backslash B_r} \frac{\left|u(y,t)-(u)_{Q_r}\right|}{|y-x_0|^{n+2s}} d y\right)^2d t \right. \\ \nonumber
    	&\left. + \fiint_{Q_r}\left|u(x,t)-(u)_{Q_r}\right|^2 dxdt + r^{2-2s-\varepsilon} \fiint_{Q_r} |f(u(x,t))|^2 dx dt \right]
    \end{align}
    \end{corollary}
    
    \begin{proof}
    	Combine \eqref{eq.caccio.ineq} with \eqref{eq.difference.estimate} and divide by $r^{n+2}$ to get average integrals. \qed
    \end{proof}
    
    \subsection{Hölder continuity of the solutions}
    
    We will now prove the Hölder continuity of solutions to \eqref{1.1}. We will do so via a Campanato argument, using the estimates we have developed to show the appropriate decays on the desired quantities. First of all, we will need the following lemma about the decay of the nonlocal integral in \eqref{eq.difference.estimate}, proved in \cite{D2024}.
    
    \begin{lemma} \label{lemma.nonlocaldecay}
    Let $0<d\leq r<1$, $g\in L^2_{loc}((0,T),H^1_{loc}(\Omega)) \cap L^\infty_{loc}((0,T), L^1_{2s}(\mathbb{R}^n))$ and $k\in\mathbb{R}$. Then if $\theta\in(0,1)$ is such that $\theta r\leq d \leq r$ and $0<\varepsilon<1-s$, then
    \begin{equation} \label{eq.obvious}
    	\fiint_{Q_r} |g(x,t)-(g)_{Q_r}|^2 dx dt \leq 4 \fiint_{Q_r} |g(x,t)-k|^2 dx dt,
    \end{equation}
    \begin{equation} \label{eq.normalinequality}
    	\fiint_{Q_d} |g(x,t)-(g)_{Q_d}|^2 dx dt \leq 4\theta^{-n-2} \fiint _{Q_r} |g(x,t)-(g)_{Q_r}|^2 dx dt,
    \end{equation}
    and
    \begin{align} \label{eq.nonlocaldecay}
    	\nonumber & \left( d^{2-\varepsilon} \int_{t_0-d^2}^{t_0} \left( \int_{\mathbb{R}^n\setminus B_d} \frac{|g(y,t)-(g)_{Q_d}|}{|x_0-y|^{n+2s}}dy\right)^2 dt \right)^{1/2} \\
    	\leq& \; c \left(\frac{d}{r}\right)^{1-\varepsilon/2} \left( r^{2-\varepsilon} \int_{t_0-r^2}^{t_0} \left( \int_{\mathbb{R}^n\setminus B_r} \frac{|g(y,t)-(g)_{Q_r}|}{|x_0-y|^{n+2s}}dy\right)^2 dt \right)^{1/2} \\
    	\nonumber&+ c d^{\frac{2-\varepsilon}{2-s}} \left( \int_d^r \left(\frac{d}{\nu}\right)^s \left( \fiint_{Q_\nu} |g(x,t)-(g)_{Q_\nu}|^2 dx dt \right)^{1/2} \frac{d\nu}{\nu} \right. \\
    	\nonumber &+ \left. \left(\frac{d}{r}\right)^s \left( \fiint_{Q_r} |g(x,t)-(g)_{Q_r}|^2 dx dt \right)^{1/2} \right),
    \end{align}
    where $c=c(n,s)$.
    \end{lemma}
    
    \begin{proof}[Proof of Theorem \ref{teo.holder}]
    	First, choose $\tilde{T} \in (0,T)$. We will prove Hölder continuity in $\Omega \times (0,\tilde{T})$, so fix $(x_0,t_0) \in \Omega \times (0,\tilde{T})$. Let $r<1$ and $\tau \leq 1/8$ such that $\overline{B}_r(x_0) \subset \Omega$ and $[t_0-r^2,t_0] \subset (0,\tilde{T})$. Then, using \eqref{eq.normalinequality}, \eqref{eq.obvious} and Lemma \ref{lemma.heatproperties} we know that
    	\begin{align*}
    		&\fiint_{Q_{\tau r}} |u(x,t)-(u)_{Q_{\tau r}}|^2 dx dt \\
    		\leq& \; c \tau^2 \fiint_{Q_{r/8}} |w-(w)_{Q_{r/8}}|^2 dx dt + c \tau^{-n-2} \fiint_{Q_{r/4}} |u-w|^2 dx dt \\
    		\leq& \; c (\tau r)^2 \fiint_{Q_{r/4}} |\nabla w|^2 dx dt + c \tau^{-n-2} \fiint_{Q_{r/4}} |u-w|^2 dx dt \\
    		\leq& \; c (\tau r)^2 \left(\fiint_{Q_{r/4}} |\nabla (u-w)|^2 dx dt + \fiint_{Q_{r/4}} |\nabla u|^2 dx dt \right) + c \tau^{-n-2} \fiint_{Q_{r/4}} |u-w|^2 dx dt.
    	\end{align*}
    	Then, using \eqref{eq.caccio.ineq} and \eqref{eq.merged.estimate}, we get
    	\begin{equation}
    		\label{eq.holder.1}
    		\begin{aligned} 
    			 &\fiint_{Q_{\tau r}} |u(x,t)-(u)_{Q_{\tau r}}|^2 dx dt \\
    			\leq& \; c(\tau^2 + r^\varepsilon \tau^{-n-2}) \left( r^{2-\varepsilon} \int_{t_0-r^2}^{t_0} \left( \int_{\mathbb{R}^n\setminus B_r} \frac{|u(y,t)-(u)_{Q_r}|}{|x_0-y|^{n+2s}} dy\right)^2 dt \right. \\
    			 &+ r^2 \fiint_{Q_r} |f(u(x,t))|^2 dx dt + \left. \fiint_{Q_r} |u(x,t)-(u)_{Q_r}|^2 dx dt   \right),
    		\end{aligned}
    	\end{equation}
    	where $c=c(n,s)$. Now, we notice that thanks to \eqref{eq.nonlocaldecay}, the nonlocal integral term decays like
    	\begin{align} \label{eq.holder.2}
    	\nonumber &(\tau r)^{2-\varepsilon} \int_{t_0-(\tau r)^2}^{t_0} \left( \int_{\mathbb{R}^n\setminus B_{\tau r}} \frac{|u(y,t)-(u)_{Q_{\tau r}}|}{|x_0-y|^{n+2s}} dy\right)^2 dt \\
    	\leq& \; c(\tau^{2-\varepsilon} + r^\varepsilon \tau^{-n-2-2s}) \left(r^{2-\varepsilon} \int_{t_0-r^2}^{t_0} \left( \int_{\mathbb{R}^n\setminus B_r} \frac{|u(y,t)-(u)_{Q_r}|}{|x_0-y|^{n+2s}} dy\right)^2 dt + \fiint_{Q_r} |u(x,t)-(u)_{Q_r}|^2 dx dt \right).
    	\end{align}
    	Similarly, the following decay is also clear
    	\begin{equation} \label{eq.holder.3}
    		(\tau r)^2 \leq \tau^{2-\varepsilon} r^2 \leq (\tau^{2-\varepsilon} + r^\varepsilon \tau^{-n-2-2s}) r^2. 
    	\end{equation}
    	Finally, we recall that we are working in interior cylinders of our parabolic domain. In particular, we are working in times $t\in(t_0-r^2,t_0)$. Using Proposition \ref{prop.semigroup.regularity}, we know that $u$ is bounded away from $\alpha$ in such cylinders and, therefore
    	\begin{equation} \label{eq.holder.4}
    		\fiint_{Q_r} |f(u(x,t))|^2 dx dt \leq \|f(u)\|^2_{L^\infty(Q_r)} \leq \|f(u)\|^2_{L^\infty(\Omega \times [0,\tilde{T}])}.
    	\end{equation}
    	Notice that this bound is finite because $u \in L^\infty(\Omega \times [0,\tilde{T}]) < \alpha$ for any $\tilde{T} \in (0,T)$ thanks to Proposition \ref{prop.semigroup.regularity} and the fact that $f \in C^{1}_{loc}([0,\alpha))$. Furthermore, this bound does not depend on $\tau$, $r$ or the center of the cylinder $(x_0,t_0)$.
    	
    	Then, adding up the inequalities \eqref{eq.holder.1}, \eqref{eq.holder.2} and \eqref{eq.holder.3} and using \eqref{eq.holder.4}, we get
    	\begin{equation}
    		 \label{eq.holder.5}
    		 \begin{aligned}
    		 	&\left( (\tau r)^{2-\varepsilon} \int_{t_0-(\tau r)^2}^{t_0} \left( \int_{\mathbb{R}^n\setminus B_{\tau r}} \frac{|u(y,t)-(u)_{Q_{\tau r}}|}{|x_0-y|^{n+2s}} dy\right)^2 dt\right)^{1/2} + \left( \fiint_{Q_{\tau r}} |u(x,t)-(u)_{Q_{\tau r}}|^2 dx dt \right)^{1/2} + \tau r \\
    		 	 \leq& \; C_0(\tau^{1-\varepsilon/2} + r^{\varepsilon/2} \tau^{-n/2-1-s}) \left[ \left( r^{2-\varepsilon} \int_{t_0-r^2}^{t_0} \left( \int_{\mathbb{R}^n\setminus B_{r}} \frac{|u(y,t)-(u)_{Q_{r}}|}{|x_0-y|^{n+2s}} dy\right)^2 dt\right)^{1/2} \right.  \\
    		 	&\left.+ \left( \fiint_{Q_{r}} |u(x,t)-(u)_{Q_{r}}|^2 dx dt \right)^{1/2} + r\right],
    		 \end{aligned}
    	\end{equation}
    	where $C_0=C_0(n,s,\|f(u)\|_{L^\infty(\Omega \times [0,\tilde{T}])})$. Notice we properly name this constant to use it on the next step. We assume without loss of generality that $C_0\geq 1$.
    	
    	Fix now some $\alpha_0 \in (0,1)$, and choose some $\alpha$ that satisfies $\alpha_0 < \alpha < 1$. Define $\alpha_1 = (1+\alpha)/2$ and take $\tau_0 = \tau_0(n,s,\alpha,\|f(u)\|_{L^\infty(\Omega \times [0,\tilde{T}])}) \leq 1/8$ such that
    	\begin{equation}\label{eq.tau.conditions}
    		C_0 \tau_0^{(1-\varepsilon/2)(1-\alpha_1)} \leq \frac{1}{2},
    	\end{equation}
    	and, for $\tau_0$ fixed, take some $r_\ast = r_\ast(n,s,\alpha,\|f(u)\|_{L^\infty(\Omega \times [0,\tilde{T}])}) \in (0,1)$ such that
    	\begin{equation} \label{eq.r.conditions}
    		C_0 r^{\varepsilon/2} \tau_0^{-n/2-1-s-(1-\varepsilon/2)\alpha_1} \leq \frac{1}{2},
    	\end{equation}
    	for every $r\leq r_\ast$. Under these conditions, \eqref{eq.holder.5} turns into
    	\begin{align} \label{eq.holder.6}
    		&\nonumber\left( (\tau_0 r)^{2-\varepsilon} \int_{t_0-(\tau_0 r)^2}^{t_0} \left( \int_{\mathbb{R}^n\setminus B_{\tau_0 r}} \frac{|u(y,t)-(u)_{Q_{\tau_0 r}}|}{|x_0-y|^{n+2s}} dy\right)^2 dt\right)^{1/2} + \left( \fiint_{Q_{\tau_0 r}} |u(x,t)-(u)_{Q_{\tau_0 r}}|^2 dx dt \right)^{1/2} + \tau_0 r \\
    		\leq& \; \tau_0^{(1-\varepsilon/2) \alpha_1} \left[ \left( r^{2-\varepsilon} \int_{t_0-r^2}^{t_0} \left( \int_{\mathbb{R}^n\setminus B_{r}} \frac{|u(y,t)-(u)_{Q_{r}}|}{|x_0-y|^{n+2s}} dy\right)^2 dt\right)^{1/2} + \left( \fiint_{Q_{r}} |u(x,t)-(u)_{Q_{r}}|^2 dx dt \right)^{1/2} + r\right],
    	\end{align}
    	for every $r\leq r_\ast$.
    	Now this is a decay estimate for our desired quantity. However, this is only valid for a fixed $\tau_0$. We now want to extend this decay to every $r\leq r_\ast$. First, let us define for convenience the following quantities for every $r\leq r_\ast$:
    \begin{align*}
    	I_1(r) =& \;\left( \int_{t_0-r^2}^{t_0} \left( \int_{\mathbb{R}^n\setminus B_{r}} \frac{|u(y,t)-(u)_{Q_{r}}|}{|x_0-y|^{n+2s}} dy\right)^2 dt\right)^{1/2},\\
    	I_2(r) =& \; \left( \fiint_{Q_{r}} |u(x,t)-(u)_{Q_{r}}|^2 dx dt \right)^{1/2}.
    \end{align*}
    Then \eqref{eq.holder.6} turns into
    \begin{equation} \label{eq.holder.7}
    	(\tau_0 r)^{1-\varepsilon/2} I_1(\tau_0 r) + I_2(\tau_0 r) + \tau_0 r \leq \tau_0^{(1-\varepsilon/2)\alpha_1} ( r^{1-\varepsilon/2} I_1(r) + I_2(r) + r),
    \end{equation}
    for every $r\leq r_\ast$. Let us now define the sharp fractional maximal operator and its truncated version for $0<\delta<1/2$:
    \begin{equation} \label{sharp.maximal.operator}
    	M(x_0,r) := \sup_{\nu \leq r}  \nu^{-(1-\varepsilon/2)\alpha} (\nu^{1-\varepsilon/2} I_1(\nu) + I_2(\nu) + \nu),
    \end{equation}
    \begin{equation} \label{truncated.maximal.operator}
    	M_\delta (x_0,r) := \sup_{\delta r \leq \nu \leq r}  \nu^{-(1-\varepsilon/2)\alpha} (\nu^{1-\varepsilon/2} I_1(\nu) + I_2(\nu) + \nu).
    \end{equation}
    
    We now consider $r\leq \rho \leq r_\ast$, multiply both sides of \eqref{eq.holder.7} by $(\tau_0 r)^{-(1-\varepsilon/2)\alpha}$ and take the supremum with respect to $r\in(\delta \rho,\rho)$ to get:
    \begin{align} \label{eq.holder.8}
    	\nonumber M_\delta (x_0, \tau_0 \rho) &\leq \tau_0^{(1-\varepsilon/2)(\alpha_1-\alpha)} \sup_{\delta \rho \leq \nu \leq \rho}  \nu^{-(1-\varepsilon/2)\alpha} (\nu^{1-\varepsilon/2} I_1(\nu) + I_2(\nu) + \nu) \\
    	&\leq \tau_0^{(1-\varepsilon/2)(1-\alpha)/2} \sup_{\delta \tau_0 \rho \leq \nu \leq \rho}  \nu^{-(1-\varepsilon/2)\alpha} (\nu^{1-\varepsilon/2} I_1(\nu) + I_2(\nu) + \nu) \\
    	\nonumber &\leq \frac{1}{2} \left(M_\delta(x_0,\tau_0 \rho) + \sup_{\tau_0 \rho \leq \nu \leq \rho}  \nu^{-(1-\varepsilon/2)\alpha} (\nu^{1-\varepsilon/2} I_1(\nu) + I_2(\nu) + \nu)\right),
    \end{align}
    where we have used condition \eqref{eq.tau.conditions}. This means that $M_\delta(x_0,\rho)$ is always finite even if $\delta\to 0$, since the second term on the right hand side is independent of $\delta$ and finite. We then assume that $\delta \leq \tau_0$, take into account that $M_\delta(x_0,r)$ is decreasing with respect to $\delta$ and use \eqref{eq.holder.8} to finally deduce that
    
    \begin{align*}
    	M_\delta(x_0,\rho) =& \; \sup_{\delta \rho \leq \nu \leq \rho}  \nu^{-(1-\varepsilon/2)\alpha} (\nu^{1-\varepsilon/2} I_1(\nu) + I_2(\nu) + \nu) \\
    	=& \; \max\left\{\sup_{\delta \rho \leq \nu \leq \tau_0\rho}  (\nu^{-(1-\varepsilon/2)\alpha} (\nu^{1-\varepsilon/2} I_1(\nu) + I_2(\nu) + \nu)),  \sup_{\tau_0 \rho \leq \nu \leq \rho} (\nu^{-(1-\varepsilon/2)\alpha} (\nu^{1-\varepsilon/2} I_1(\nu) + I_2(\nu) + \nu)) \right\} \\
    	=&\; \max\left\{M_{\delta/\tau_0}(x_0,\tau_0\rho), \sup_{\tau_0 \rho \leq \nu \leq \rho} \nu^{-(1-\varepsilon/2)\alpha} (\nu^{1-\varepsilon/2} I_1(\nu) + I_2(\nu) + \nu) \right\} \\
    	\leq&\; \max\left\{M_{\delta}(x_0,\tau_0\rho), \sup_{\tau_0 \rho \leq \nu \leq \rho} \nu^{-(1-\varepsilon/2)\alpha} (\nu^{1-\varepsilon/2} I_1(\nu) + I_2(\nu) + \nu) \right\} \\
    	\leq& \; c \sup_{\tau_0 \rho \leq \nu \leq \rho} \nu^{-(1-\varepsilon/2)\alpha} (\nu^{1-\varepsilon/2} I_1(\nu) + I_2(\nu) + \nu) \\
    	\leq& \; c (\tau_0\rho)^{-(1-\varepsilon/2)\alpha} \sup_{\tau_0 \rho \leq \nu \leq \rho} (\nu^{1-\varepsilon/2} I_1(\nu) + I_2(\nu) + \nu) \\
    	\leq& \; \frac{c}{\rho^{(1-\varepsilon/2)\alpha}} \sup_{\tau_0 \rho \leq \nu \leq \; \rho} (\nu^{1-\varepsilon/2} I_1(\nu) + I_2(\nu) + \nu),
    \end{align*}
    with $c=c(n,s,\alpha, \|f(u)\|_{L^\infty(\Omega\times[0,\tilde{T}])})$. Since nothing on the left hand side depends on $\delta$, we can take $\delta \to 0$, which yields
    \begin{equation*}
    	M(x_0,\rho) \leq \frac{c}{\rho^{(1-\varepsilon/2)\alpha}} \sup_{\tau_0 \rho \leq \nu \leq \rho} (\nu^{1-\varepsilon/2} I_1(\nu) + I_2(\nu) + \nu).
    \end{equation*}
    Now using the fact that $\tau\rho \leq \nu\leq \rho$, and $\tau_0$ is fixed, we can use Lemma \ref{lemma.nonlocaldecay} to get that
    \begin{equation*}
    	M(x_0,\rho) \leq \frac{c}{\rho^{(1-\varepsilon/2)\alpha}} (\rho^{1-\varepsilon/2} I_1(\rho) + I_2(\rho) + \rho).
    \end{equation*}
    Finally, we just use the definition of $M(x_0,r)$, \eqref{sharp.maximal.operator}, to get that, for every $0<r\leq \rho \leq r_\ast$:
    \begin{equation} \label{need.to.prove}
    	r^{1-\varepsilon/2} I_1(r) + I_2(r) + r \leq c \left(\frac{r}{\rho}\right)^{(1-\varepsilon/2)\alpha} (\rho^{1-\varepsilon/2} I_1(\rho) + I_2(\rho) + \rho),
    \end{equation}
     which is exactly the decay we need to close the argument. In order to recover $\alpha_0$, we just choose $0<\varepsilon<\min\{2-2\alpha_0,1-s\}$ and $\alpha=\frac{2\alpha_0}{2-\varepsilon}$. Taking $\rho=r_\ast$ as well, we end up with
    \begin{equation*}
    	r^{1-\varepsilon/2} I_1(r) + I_2(r) + r \leq c \left(\frac{r}{r_\ast}\right)^{\alpha_0} (r_\ast^{1-\varepsilon/2} I_1(r_\ast) + I_2(r_\ast) + r_\ast),
    \end{equation*}
    which implies
    \begin{equation*}
    	\fiint_{Q_r} |u-(u)_{Q_r}|^2 dx dt \leq c\left(\frac{r}{r_\ast}\right)^{2\alpha_0} (r_\ast^{2-\varepsilon} I_1^2(r_\ast) + I_2^2(r_\ast) + r_\ast^2).
    \end{equation*}
    We can then use \eqref{eq.obvious} and the fact that $\varepsilon < 1-s$ to see that
    \begin{equation*}
    	r_\ast^{2-\varepsilon} I_1^2(r_\ast) + I_2^2(r_\ast) \leq \frac{c}{r_\ast^{n+2}}\left( \text{Tail}_\infty(u,x_0,r_\ast,t_0-r_\ast^2,t_0)^2 + \|u\|^2_{L^2([t_0-r_\ast^2,t_0],L^2(B_{r_\ast}))}\right).
    \end{equation*}
    Now notice that, for any solution $u$ of our problem, using the fact that $u\in C^1((0,T),L^2(\Omega))$ and $u\equiv 0$ outside of $\Omega$, we get using Hölder's inequality:
    \begin{align} \label{eq.tailestimate1}
    	\nonumber \text{Tail}_\infty & (u,x_0,r_\ast,t_0-r_\ast^{2},t_0) \\
    	\nonumber \leq& \; r_\ast^{2s} \sup_{t\in(t_0-r^{2},t_0)} \left(  \int_{\mathbb{R}^n\setminus B_{r_\ast}(x_0)}
    	\hspace{-.5cm} |x_0-y|^{-2n-4s}  dy \right)^{1/2} \left(  \int_{\Omega\setminus B_{r_\ast}(x_0)}
    	\hspace{-.5cm} |u(y,t)|^2 dy \right)^{1/2} \\
    	\leq& \; r_\ast^{2s} \sup_{t\in(t_0-r^{2},t_0)} C r_\ast^{-\frac{n}{2}-2s} \|u(\cdot,t)\|_{L^2(\Omega)} \\
    	\nonumber \leq& \; C r_\ast^{-\frac{n}{2}} \|u\|_{L^\infty([0,\tilde{T}], L^2(\Omega))},
    \end{align}
    where $C=C(n,s)$. Therefore
	\begin{equation*}
    	r_\ast^{2-\varepsilon} I_1^2(r_\ast) + I_2^2(r_\ast) \leq \frac{c}{r_\ast^{n+2}}\left( r_\ast^{-n} \|u\|^2_{L^\infty([0,\tilde{T}], L^2(\Omega))} + \|u\|^2_{L^2(\Omega \times[0,\tilde{T})}\right),
    \end{equation*}
    which is a finite quantity that does ot depend on $\tau$, $r$, or the center of the cylinder $(x_0,t_0)$. Then, we conclude
    \begin{equation} \label{eq.holder.final}
    	\fiint_{Q_r} |u-(u)_{Q_r}|^2 dx dt \leq  c\left(\frac{r}{r_\ast}\right)^{2\alpha_0},
    \end{equation}
    where $c=c(n,s,\alpha_0,\|f(u)\|_{L^\infty(\Omega \times [0,\tilde{T}])}, \|u\|_{L^2(\Omega \times[0,\tilde{T})}, \|u\|_{L^\infty([0,\tilde{T}], L^2(\Omega))})$. We then conclude by the integral characterization of Campanato spaces that $u \in C^{\alpha_0}_{loc} (\Omega \times (0,\tilde{T}))$. Since this is valid for every $\tilde{T} \in (0,T)$, then $u \in C^{\alpha_0}_{loc} (\Omega \times (0,T))$ \qed
    \end{proof}

    \begin{remark}
    	Notice that, since the classical solution of \eqref{1.1} is the same as that of \eqref{1.1.truncated}, then Lemma \ref{initial.data.comparison}, Lemma \ref{lemma.comparison.abstract} and Corollary \ref{cor.positivity} are also valid for the solutions of \eqref{1.1}.
    \end{remark}
    
    \section{$L^2-L^\infty$ estimate for the equation with a zeroth order term} \label{sec.l2linfty}
	
	We will now study system \eqref{eq.zeroth.order}. Although system \eqref{eq.zeroth.order} is not exactly of the form \eqref{1.1}, it is easy to adapt the arguments of Lemma \ref{lemma.truncated.solution} by noticing that the Nemytskii operator generated by $a(x,t)u(x,t)$ is globally Lipschitz in $L^2(\Omega)$ and the function $t \rightarrow a(\cdot,t)u(\cdot,t)$ is continuous in $L^2(\Omega)$. Therefore, we know that we have a unique classical solution
    \begin{equation*}
    	v \in C((0,T^\prime), D(\mathcal{L})) \cap C([0,T^\prime),H_0^1(\Omega)) \cap C^1((0,T^\prime),L^2(\Omega)),
    \end{equation*}
    of system \eqref{eq.zeroth.order}.
    
    Our objective is to find an $L^2-L^\infty$ estimate for equation \eqref{eq.zeroth.order}. To do so, we will follow similar arguments to those on the paper \cite{FSZ2022}, where they proved such estimates for the purely linear case. We will pay special attention to the constants that appear in the estimates, and their dependence on $r$ and $\tau$. The main result of this section is the following.

    \subsection{Another Caccioppoli estimate}
    
    The first step is to find a Caccioppoli inequality for system \eqref{eq.zeroth.order}. This turns out as follows.
    
    \begin{lemma} \label{lemma.cacciopoli.zeroth}
    	Let $v$ be the classical solution of equation \eqref{eq.zeroth.order}. Let $B_r\equiv B_r(x_0)$ be such that $\overline{B}_r \subset \Omega$ and $0<t_1<t_2$, $l>0$ satisfy $[t_1-l,t_2] \subset (0,T^\prime)$. Then for any nonnegative functions $\psi \in C_0^\infty(B_r)$ and $\eta\in C^\infty(\mathbb{R})$ such that $\eta \equiv 0$ if $t\leq t_1-l$ and $\eta\equiv 1$ if $t\geq t_1$, it holds that
    	\begin{align} \label{eq.cacciopoli.zeroth}
    		\nonumber\int_{t_1-l}^{t_2} & \int_{B_r} |\nabla w_+(x,t)|^2\psi^2(x)\eta^2(t) dx dt + \sup_{t\in(t_1,t_2)} \int_{B_r} w_+^2(x,t)\psi^2(x) dx \\ 
    		\nonumber &+ \int_{t_1-l}^{t_2} \int_{B_r} \int_{B_r} \frac{|w_+(x,t)\psi(x)-w_+(y,t)\psi(y)|^2 \eta^2(t)}{|x-y|^{n+2s}} dxdydt \\
    		\leq& \,\, C \int_{t_1-l}^{t_2} \int_{B_r} |\nabla \psi(x)|^2 w_+^2(x,t)\eta^2(t) dx dt \\
    		\nonumber &+ C \int_{t_1-l}^{t_2} \int_{B_r} \int_{B_r} \frac{\max\{w_+(x,t),w_+(y,t)\}^2 }{|x-y|^{n+2s}}|\psi(x)-\psi(y)|^2 \eta^2(t) dxdydt \\
    		\nonumber &+ C \sup_{\substack{t\in(t_1-l,t_2)\\ x\in\operatorname{supp}(\psi)}} \int_{\mathbb{R}^n\setminus B_r} \frac{w_+(y,t)}{|x-y|^{n+2s}} dy \int_{t_1-l}^{t_2} \int_{B_r} w_+(x,t)\psi^2(x)\eta^2(t) dxdt \\
    		\nonumber &+C \int_{t_1-l}^{t_2} \int_{B_r} w_+^2(x,t) \psi^2(x) \eta(t) |\partial_t \eta(t)| dx dt \\
    		\nonumber &+  C \int_{t_1-l}^{t_2} \int_{B_r} w_+^2(x,t) \psi^2(x) \eta^2(t) dx dt + C k \int_{t_1-l}^{t_2} \int_{B_r} w_+(x,t) \psi^2(x) \eta^2(t) dx dt
    	\end{align}
    	where $w_+(x,t) = (v(x,t)-k)_+$ for some $k>0$ and $C=C(n,s, \|a\|_{L^\infty(\Omega\times[t_1-l,t_2]})$.
    \end{lemma}
    \begin{proof}
    	This proof is very similar to that of Proposition \ref{prop.cacciopoli}. Again, the first step of this argument should be the regularization procedure with respect to the time variable, which we will omit again. Take $\varphi(x,t) = w_+(x,t) \psi^2(x) \eta^2(t)$ as a test function in equation \eqref{eq.zeroth.order} to get:
    	\begin{align*}
    		&\int_{t_1-l}^{t_2} \int_{B_r} \partial_t v(x,t) w_+ \psi^2 \eta^2(x,t) dx dt + \int_{t_1-l}^{t_2} \int_{B_r} \nabla v(x,t) \nabla (w_+\psi^2\eta^2(x,t)) dxdt \\
    		&+ \frac{C_{n,s}}{2} \int_{t_1-l}^{t_2} \int_{B_r} \int_{B_r} \frac{(v(x,t)-v(y,t))(w_+ \psi^2\eta^2(x,t) - w_+ \psi^2\eta^2(y,t))}{|x-y|^{n+2s}} dxdydt \\
    		&+ \frac{C_{n,s}}{2} \int_{t_1-l}^{t_2} \int_{\mathbb{R}^n \setminus B_r} \int_{B_r} \frac{(v(x,t)-v(y,t))w_+ \psi^2\eta^2(x,t)}{|x-y|^{n+2s}} dxdydt \\
    		\leq& \; \int_{t_1-l}^{t_2} \int_{B_r} a(x,t)v(x,t) w_+ \psi^2 \eta^2(x,t) dx dt.
    	\end{align*}
    	The left hand side term can be treated in a very similar way as in the homogeneous problem. The only difference is that the nonlocal integral over $\mathbb{R}^n\setminus B_r$ will be estimated by a different quantity that will be more practical for our purposes, see \cite[Lemma 3.1]{FSZ2022} for details. Apart from that, we only need to estimate the new integral that spawns from the zeroth order term. Noticing that $v w_+ = (v-k)(v-k)_+ + k(v-k)_+ = w_+^2 + k w_+$, we end up with
    	\begin{align*}
    		\int_{t_1-l}^{t_2}& \int_{B_r} a(x,t)v(x,t) w_+(x,t) \psi^2(x) \eta^2(t) dx dt \\
    		\leq& \; \|a\|_{L^\infty(\Omega\times[t_1-l,t_2])} \int_{t_1-l}^{t_2} \int_{B_r} v(x,t) w_+(x,t) \psi^2(x) \eta^2(t) dx dt \\
    		=& \; \|a\|_{L^\infty(\Omega\times[t_1-l,t_2])} \int_{t_1-l}^{t_2} \int_{B_r} w_+^2(x,t) \psi^2(x) \eta^2(t) dx dt  \\
    		&+ \|a\|_{L^\infty(\Omega\times[t_1-l,t_2])} k \int_{t_1-l}^{t_2} \int_{B_r} w_+(x,t) \psi^2(x) \eta^2(t) dx dt,
    	\end{align*}
    	which leads to the desired result. \qed
    \end{proof}

    \subsection{Moser iteration for local boundedness}
    
    We are going to introduce some notation that will be useful to enable the Moser iteration we will need to prove the $L^2-L^\infty$ estimate. First, for some $\tau \in (0,1)$ and $r\in(0,1)$, set
    \begin{equation*}
    	r_0 = r, \;\; r_j = \tau r + 2^{-j}(1-\tau) r, \;\; \tilde{r}_j = \frac{r_j+r_{j+1}}{2}, \;\; j=0,1,2\dots,
    \end{equation*}    
    and
    \begin{align*}
    	Q_j &= B_j \times I_j = B_{r_j}(x_0) \times (t_0-r_j^2,t_0), \;\; j=0,1,2\dots, \\
    	\tilde{Q}_j &= \tilde{B}_j \times \tilde{I}_j = B_{\tilde{r}_j}(x_0) \times (t_0-\tilde{r}_j^2,t_0), \;\; j=0,1,2\dots.
    \end{align*}
    Notice that the following relations are satisfied:
    \begin{equation*}
    	r_j > \tilde{r}_j > r_{j+1}, \;\; B_{j+1} \subset \tilde{B}_j \subset B_j, \;\; I_{j+1} \subset \tilde{I}_j \subset I_j, \;\; j=0,1,2\dots
    \end{equation*}
    Now choose
    \begin{equation*}
    	\tilde{k} \geq \frac{1}{2} \text{Tail}_\infty (v_+, x_0, \tau r, t_0-r^{2}, t_0),
    \end{equation*}
    and denote
    \begin{equation*}
		k_j = (1-2^{-j})\tilde{k}, \;\; \tilde{k}_j = \frac{k_{j+1}+k_j}{2}, \;\; j=0,1,2\dots,   	
    \end{equation*}
    which in turn satisfy the relations
    \begin{equation*}
    	\tilde{k} \geq k_{j+1} \geq \tilde{k}_j \geq k_j \geq 0, \;\; j=0,1,2\dots
    \end{equation*}
    Finally, let
    \begin{equation*}
    	w_j = (v-k_j)_+, \;\; \tilde{w}_j = (v-\tilde{k}_j)_+, \;\; j=0,1,2\dots
    \end{equation*}
    Notice that these functions satisfy the following useful inequality for every $0\leq \alpha < \beta$ and every $(x,t) \in \Omega \times [0,T)$:
    \begin{equation} \label{trivial.but.useful}
    	\tilde{w}^\alpha_j(x,t) \leq w^\alpha_j(x,t) \leq \frac{C 2^{(\beta-\alpha)j}}{\tilde{k}^{\beta-\alpha}} w_j^\beta(x,t).
    \end{equation}
    
    Let us also state the following parabolic Sobolev inequality that will be very useful to us later on. For the proof, see \cite{AF2003}.
    \begin{lemma} \label{lemma.Linfty.estimate}
    	For every $u\in L^2((t_1,t_2), H^1(B_r)) \cap L^\infty((t_1,t_2),L^2(B_r))$ with $0<t_1<t_2$, it holds that
    	\begin{align}
    		\nonumber \int_{t_1}^{t_2} & \fint_{B_r} |u(x,t)|^{2(1+2/n)} dx dt \\
    		\leq& \; C \left(r^2 \int_{t_1}^{t_2} \fint_{B_r} |\nabla u(x,t)|^2 dx dt + \int_{t_1}^{t_2} \fint_{B_r} |u(x,t)|^2 dx dt \right) \left( \sup_{t\in(t_1,t_2)} \fint_{B_r} |u(x,t)|^2 dx \right)^{2/n}.
    	\end{align}
    \end{lemma}
    
    We will now particularize our Caccioppoli inequality from Lemma \ref{lemma.cacciopoli.zeroth} to some appropriate cutoff functions that will help us get some inequality to perform the recursive iteration.
    
    \begin{lemma} \label{lemma.Cacciopoli.improved.zeroth}
    	Let $v$ be the classical solution of equation \eqref{eq.zeroth.order} and $\tilde{T} \in (0,T^\prime)$. Let $(x_0,t_0) \in \Omega \times (0,\tilde{T})$, $0<r<1$, and $Q_r = B_r(x_0)\times (t_0-r^2,t_0)$ such that $\overline{B}_r(x_0) \subset \Omega$ and $[t_0-r^2,t_0] \subset (0,\tilde{T})$. Then it holds that
    	\begin{align*}
    		\int_{I_{j+1}}& \fint_{B_{j+1}} |\nabla \tilde{w}_j(x,t)|^2 dx dt + \sup_{t\in I_{j+1}} \fint_{B_{j+1}} \tilde{w}_j^2(x,t) dx + \int_{I_{j+1}} \int_{B_{j+1}} \fint_{B_{j+1}} \frac{|\tilde{w}_j(x,t)-\tilde{w}_j(y,t)|^2}{|x-y|^{n+2s}} dxdydt \\
    		\leq& \, \frac{C}{r^2(1-\tau)^2 } \left( 2^{2j+2} + \frac{2^{(n+2s+1)j}}{\tau^2(1-\tau)^{n+2s-2}} \right) \int_{I_j} \fint_{B_j} w_j^2(x,t) dx dt,
    	\end{align*}
    	for every $j\in\mathbb{N}$, where $C=C(n,s,\|a\|_{L^\infty(\Omega\times[0,\tilde{T}])})$.
    \end{lemma}
    
    \begin{proof}
    	We need to choose appropriate cutoff functions to make \eqref{eq.cacciopoli.zeroth} fit our purposes. Take $\psi_j \in C_0^\infty (\tilde{B}_j)$ and $\eta_j \in C_0^\infty(\tilde{I}_j)$ such that
    	\begin{equation*}
    		 0 \leq \psi_j \leq 1, \;\; |\nabla \psi_j| \leq \frac{C 2^j}{(1-\tau)r}, \; \text{ in } \tilde{B}_j, \;\;\; \psi \equiv 1 \text{ in } B_{j+1},
    	\end{equation*}
    	and
    	\begin{equation*}
    		 0 \leq \eta_j \leq 1, \;\; |\partial_t \eta_j| \leq \frac{C 2^{2j}}{(1-\tau)^2 r^2}, \; \text{ in } \tilde{I}_j, \;\;\; \eta \equiv 1 \text{ in } I_{j+1}.
    	\end{equation*}
    	Now take $B_r = B_j$, $t_2=t_0$, $t_1=t_0-r_{j+1}^2$, $l=\tilde{r}_j^2 -r_{j+1}^2$ and $k=\tilde{k}_j$ in Lemma \ref{lemma.cacciopoli.zeroth}. Then we get
    	\begin{align*} 
    		\int_{\tilde{I}_j} & \int_{B_j} |\nabla \tilde{w}_j(x,t)|^2\psi_j^2(x)\eta_j^2(t) dx dt + \sup_{t\in I_{j+1}} \int_{B_j} \tilde{w}_j^2(x,t)\psi_j^2(x) dx \\ 
    		&+ \int_{\tilde{I}_j} \int_{B_j} \int_{B_j} \frac{|\tilde{w}_j(x,t)\psi_j(x)-\tilde{w}_j(y,t)\psi_j(y)|^2 \eta_j^2(t)}{|x-y|^{n+2s}} dxdydt \\
    		\leq& \,\, C \int_{\tilde{I}_j} \int_{B_j} |\nabla \psi_j(x)|^2 \tilde{w}_j^2(x,t)\eta_j^2(t) dx dt \\
    		&+ C \int_{\tilde{I}_j} \int_{B_j} \int_{B_j} \frac{\max\{\tilde{w}_j(x,t),\tilde{w}_j(y,t)\}^2 |\psi_j(x)-\psi_j(y)|^2 \eta_j^2(t)}{|x-y|^{n+2s}} dxdydt \\
    		&+ C \sup_{t\in\tilde{I}_j; x\in supp(\psi_j)} \int_{\mathbb{R}^n\setminus B_j} \frac{\tilde{w}_j(y,t)}{|x-y|^{n+2s}} dy \int_{\tilde{I}_j} \int_{B_j} \tilde{w}_j(x,t)\psi_j^2(x)\eta_j^2(t) dxdt \\
    		&+C \int_{\tilde{I}_j} \int_{B_j} \tilde{w}_j^2(x,t) \psi_j^2(x) \eta_j(t) |\partial_t \eta_j(t)| dx dt \\
    		&+  C \int_{\tilde{I}_j} \int_{B_j} \tilde{w}_j^2(x,t) \psi_j^2(x) \eta_j^2(t) dx dt + C \tilde{k}_j \int_{\tilde{I}_j} \int_{B_j} \tilde{w}_j(x,t) \psi_j^2(x) \eta_j^2(t) dx dt \\
    		=& \; \mathcal{I}_1 + \mathcal{I}_2 + \mathcal{I}_3 + \mathcal{I}_4 + \mathcal{I}_5 + \mathcal{I}_6.
    	\end{align*}
    	Now let us estimate each integral using the properties of $\psi_j$, $\eta_j$ and \eqref{trivial.but.useful}. The estimates of $\mathcal{I}_1$ to $\mathcal{I}_4$ are the same as the ones in \cite[Lemma 4.1]{FSZ2022}:
    	\begin{equation} \label{eq.I1}
    		\mathcal{I}_1 \leq \frac{C 2^{2j}}{(1-\tau)^2 r^2} \int_{I_j} \int_{B_j} \tilde{w}_j^2(x,t) dx dt \leq \frac{C 2^{2j}}{(1-\tau)^2 r^2} \int_{I_j} \int_{B_j} w_j^2(x,t) dx dt,
    	\end{equation}
    	\begin{equation} \label{eq.I2}
    		\mathcal{I}_2 \leq \frac{C 2^{2j}}{(1-\tau)^2 r^{2s}} \int_{I_j} \int_{B_j} w_j^2(x,t) dx dt,
    	\end{equation}
    	\begin{equation} \label{eq.I3}
    		\mathcal{I}_3 \leq \frac{C 2^{(n+2s+1)j}}{(1-\tau)^{n+2s}(\tau r)^{2s}} \int_{I_j} \int_{B_j} w_j^2(x,t) dx dt,
    	\end{equation}
    	and
    	\begin{equation} \label{eq.I4}
    		\mathcal{I}_4 \leq \frac{C 2^{2j}}{(1-\tau)^2 r^2} \int_{I_j} \int_{B_j} w_j^2(x,t) dx dt.
    	\end{equation}
    	The two remaining integrals, which are the new ones particular to our problem, are also easy to deal with
    	\begin{equation} \label{eq.I5}
    		\mathcal{I}_5 \leq C \int_{I_j} \int_{B_j} \tilde{w}_j^2(x,t) dx dt \leq C\int_{I_j} \int_{B_j} w_j^2(x,t) dx dt,
    	\end{equation}
    	and
    	\begin{equation} \label{eq.I6}
    		\mathcal{I}_6 \leq C \tilde{k}_j \int_{I_j} \int_{B_j} \tilde{w}_j(x,t) dx dt \leq C 2^j \int_{I_j} \int_{B_j} w_j^2(x,t) dx dt.
    	\end{equation}
    	
    	On the other hand, the left hand side can easily be dealt with by restricting the integral domains to $I_{j+1}$ and $B_{j+1}$, where both cutoffs are equal to $1$.
    	\begin{align} \label{eq.righthandside.zeroth}
    		\nonumber \int_{\tilde{I}_j} & \int_{B_j} |\nabla \tilde{w}_j(x,t)|^2\psi_j^2(x)\eta_j^2(t) dx dt + \sup_{t\in I_{j+1}} \int_{B_j} \tilde{w}_j^2(x,t)\psi_j^2(x) dx \\ 
    		&+ \int_{\tilde{I}_j} \int_{B_j} \int_{B_j} \frac{|\tilde{w}_j(x,t)\psi_j(x)-\tilde{w}_j(y,t)\psi_j(y)|^2 \eta_j^2(t)}{|x-y|^{n+2s}} dxdydt \\
    		\nonumber \geq& \; \int_{I_{j+1}} \int_{B_{j+1}} |\nabla \tilde{w}_j(x,t)|^2 dx dt + \sup_{t\in I_{j+1}} \int_{B_{j+1}} \tilde{w}_j^2(x,t) dx + \int_{I_{j+1}} \int_{B_{j+1}} \int_{B_{j+1}} \frac{|\tilde{w}_j(x,t)-\tilde{w}_j(y,t)|^2}{|x-y|^{n+2s}} dxdydt.
    	\end{align}
    	
    	Now, merging \eqref{eq.righthandside.zeroth} with \eqref{eq.I1}-\eqref{eq.I6} and dividing by $|B_j| \sim r^{n}$, we get
    	\begin{align*}
    		\int_{I_{j+1}}& \fint_{B_{j+1}} |\nabla \tilde{w}_j(x,t)|^2 dx dt + \sup_{t\in I_{j+1}} \fint_{B_{j+1}} \tilde{w}_j^2(x,t) dx + \int_{I_{j+1}} \int_{B_{j+1}} \fint_{B_{j+1}} \frac{|\tilde{w}_j(x,t)-\tilde{w}_j(y,t)|^2}{|x-y|^{n+2s}} dxdydt \\
    		\leq& \, \frac{C}{r^2(1-\tau)^2} \left( 2^{2j} + 2^{2j} r^{2-2s} + \frac{2^{(n+2s+1)j}}{\tau^{2s} r^{2s-2} (1-\tau)^{n+2s-2}} + r^2 (1-\tau)^2 + 2^{j} r^2(1-\tau)^2  \right) \int_{I_j} \fint_{B_j} w_j^2(x,t) dx dt.
    	\end{align*}
    	Finally, manipulating the term inside the parenthesis, bearing in mind that $0<\tau,r<1$:
    	\begin{align*}
    		2^{2j} &+ 2^{2j} r^{2-2s} + \frac{2^{(n+2s+1)j}}{\tau^{2s} r^{2s-2}(1-\tau)^{n+2s-2}} + r^2 (1-\tau)^2 + 2^{j} r^2(1-\tau)^2 \\
    		\leq& \; 2^{2j} \cdot 4 + \frac{2^{(n+2s+1)j}}{\tau^{2s} r^{2s-2}(1-\tau)^{n+2s-2}} \\
    		\leq& \;  2^{2j+2} + \frac{2^{(n+2s+1)j}}{\tau^2(1-\tau)^{n+2s-2}},
    	\end{align*}
    	which leads to the desired result. \qed
    \end{proof}
    
    \begin{lemma} \label{lemma.final.Cacciopoli.recursion}
    	Let $v$ be the classical solution of equation \eqref{eq.zeroth.order}, let $\tilde{T} \in (0,T^\prime)$ and $\kappa=1+2/n$. Let $(x_0,t_0) \in \Omega \times (0,\tilde{T})$, $0<r<1$, and $Q_r = B_r(x_0)\times (t_0-r^2,t_0)$ such that $\overline{B}_r(x_0) \subset \Omega$ and $[t_0-r^2,t_0] \subset (0,\tilde{T})$. Then it holds that
    	\begin{align*}
    		\int_{I_{j+1}} & \fint_{B_{j+1}} w_{j+1}^2(x,t) dx dt \\
    		\leq& \, \frac{C 2^{(n+4)\kappa j}}{r^{\frac{4}{n\kappa}}(1-\tau)^2 \tilde{k}^{\frac{4}{n\kappa}}} \left( 1 + \frac{1}{\tau^2 (1-\tau)^{(n+2s-2)}} \right) \left( \int_{I_j} \fint_{B_j} w_j^2(x,t) dx dt \right)^{1+\frac{2}{\kappa n}},
    	\end{align*}
    	for every $j\in\mathbb{N}$, where $C=C(n,s,\|a\|_{L^\infty(\Omega\times[0,\tilde{T}])})$.
    \end{lemma}
    
    \begin{proof}
    	First, we use Hölder's inequality to get that
    	\begin{align} \label{eq.holder.inequality}
    		\nonumber \int_{I_{j+1}} & \fint_{B_{j+1}} w_{j+1}^2(x,t) dx dt \leq \int_{I_{j+1}} \fint_{B_{j+1}} \tilde{w}_j^2(x,t) dx dt \\
    		\leq& \; \left(\int_{I_{j+1}} \fint_{B_{j+1}} \tilde{w}_j^{2\kappa}(x,t) dx dt \right)^{\frac{1}{\kappa}} \left(\int_{I_{j+1}} \fint_{B_{j+1}} \chi_{\{v \geq \tilde{k}_j\}}(x,t) dx dt \right)^{1-\frac{1}{\kappa}}.
    	\end{align}
    	Noticing that $\chi_{\{v \geq \tilde{k}_j\}}(x,t) = (\tilde{w}_j(x,t))^0$, we can use \eqref{trivial.but.useful} to obtain
    	\begin{equation} \label{eq.characteristic.inequality}
    		\int_{I_{j+1}} \fint_{B_{j+1}} \chi_{\{v \geq \tilde{k}_j\}}(x,t) dx dt \leq \frac{C 2^{2j}}{\tilde{k}^2} \int_{I_{j}} \fint_{B_{j}} w_j^2 (x,t) dx dt.
    	\end{equation}
    	
    	Now, using Lemma \ref{lemma.Linfty.estimate} and Lemma \ref{lemma.Cacciopoli.improved.zeroth} and carefully performing all calculations, we can estimate
    	\begin{align*}
    		\int_{I_{j+1}} & \fint_{B_{j+1}} \tilde{w}_j^{2\kappa} (x,t) dx dt \\
    		\leq& \; C \left(r^2 \int_{I_{j+1}} \fint_{B_{j+1}} |\nabla \tilde{w}_j(x,t)|^2 dx dt + \int_{I_{j+1}} \fint_{B_{j+1}} |\tilde{w}_j(x,t)|^2 dx dt \right)\left( \sup_{t\in I_{j+1}} \fint_{B_{j+1}} |\tilde{w}_j (x,t)|^2 dx \right)^{2/n} \\
    		\leq& \; \frac{C 2^{(n+4)\kappa j}}{r^{\frac{4}{n}} (1-\tau)^{2\kappa}} \left(1+\frac{1}{\tau^{2\kappa} (1-\tau)^{(n+2s-2)\kappa}} \right) \left(\int_{I_j} \fint_{B_j} w_j^2(x,t) dx dt \right)^{\kappa}.
    	\end{align*}
    	Combining this last inequality with \eqref{eq.holder.inequality} and \eqref{eq.characteristic.inequality}, we obtain the desired result. \qed
    \end{proof}
    
    Finally, we will make use of the following lemma from \cite{D1993} regarding sequences of positive numbers
    
    \begin{lemma} \label{lemma.sequence}
    	Let $\{Y_j\}_{j=0}^\infty$ be a sequence of positive numbers such that
    	\begin{equation*}
    		Y_0 \leq K^{-\frac{1}{\delta}} b^{-\frac{1}{\delta^2}}, \; \text{ and } \; Y_{j+1} \leq K b^j Y_j^{1+\delta}, \;\; j=0,1,2\dots,
    	\end{equation*}
    	for some $K,b>1$ and $\delta>0$. Then we have that $\lim_{j\to\infty} Y_j = 0$.
    \end{lemma}
    
    We can now prove the main theorem of this section.
    
    \begin{proof}[Proof of Theorem \ref{teo.l2.linfty}]
		Let us denote
		\begin{equation*}
			Y_j = \fiint_{Q_j} (v-k_j)^2_+ dx dt, \;\; j=0,1,2\dots
		\end{equation*}
		This transforms the result of Lemma \ref{lemma.final.Cacciopoli.recursion} into
		\begin{equation*}
			Y_{j+1} \leq C \frac{2^{(n+4)\kappa j}}{\tilde{k}^{\frac{4}{n\kappa}}} \left( \frac{1}{(1-\tau)^2} + \frac{1}{\tau^2(1-\tau)^{n+2s}} \right) Y_{j}^{1+\frac{2}{n\kappa}},
		\end{equation*}
		for any $j\in\mathbb{N}$. We can now apply Lemma \ref{lemma.sequence} if we choose
		\begin{align*}
			\tilde{k} \geq \max\Bigg\{ & \text{Tail}_{\infty}(v,x_0,\tau r, t_0-r^{2},t_0), C r^{-n/2-1} \left(\frac{1}{(1-\tau)^{\frac{n}{2}+1}} + \frac{1}{\tau^{\frac{n}{2}+1} (1-\tau)^{(\frac{n}{2}+s)(\frac{n}{2}+1)}} \right) \|v\|_{L^2(Q_r)} \Bigg\},
		\end{align*}
		so that we can conclude $\lim_{j\to\infty} Y_j = 0$. This means that $\|v\|_{L^\infty (Q_{\tau r})} \leq \tilde{k}$ for any applicable $\tilde{k}$. Therefore, choosing the smallest one possible, one gets:
		\begin{align} \label{eq.infinity.1}
			\|v\|_{L^\infty (Q_{\tau r})} \leq& \; \text{Tail}_{\infty}(v,x_0,\tau r, t_0-r^{2},t_0) + C r^{-n/2-1} \left(\frac{1}{(1-\tau)^{\frac{n}{2}+1}} + \frac{1}{\tau^{\frac{n}{2}+1} (1-\tau)^{(\frac{n}{2}+s)(\frac{n}{2}+1)}} \right) \|v\|_{L^2(Q_r)}
		\end{align}

    Now, following the same argument of \eqref{eq.tailestimate1}, we get:
    \begin{align} \label{eq.tailestimate}
    	\nonumber \text{Tail}_\infty & (v,x_0,\tau r,t_0-r^{2},t_0) \\
    	\nonumber \leq& \; (\tau r)^{2s} \sup_{t\in(t_0-r^{2},t_0)} \left(  \int_{\mathbb{R}^n\setminus B_{\tau r}(x_0)}
    	\hspace{-.5cm} |x_0-y|^{-2n-4s}  dy \right)^{1/2} \left(  \int_{\Omega\setminus B_{\tau r}(x_0)}
    	\hspace{-.5cm} |v(y,t)|^2 dy \right)^{1/2} \\
    	\leq& \; (\tau r)^{2s} \sup_{t\in(t_0-r^{2},t_0)} C (\tau r)^{-\frac{n}{2}-2s} \|v(\cdot,t)\|_{L^2(\Omega)} \\
    	\nonumber \leq& \; C (\tau r)^{-\frac{n}{2}} \|v\|_{L^\infty([t_0-r^2,t_0], L^2(\Omega))},
    \end{align}
    where $C=C(n,s)$. Then, merging \eqref{eq.infinity.1} and \eqref{eq.tailestimate}, we get		
		\begin{align*}
		\|v\|_{L^\infty(Q_{\tau r})} \leq& \; \; C((\tau r)^{-\frac{n}{2}}\|v\|_{L^\infty([t_0-r^2,t_0], L^2(\Omega))} +r^{-\frac{n}{2}-1} \tau^{-\frac{n}{2}-1} (1-\tau)^{-(\frac{n}{2}+s)(\frac{n}{2}+1)} \|v\|_{L^2(Q_r)}) \\
			\leq& \; C((\tau r)^{-\frac{n}{2}} + r^{-\frac{n}{2}} \tau^{-\frac{n}{2}-1} (1-\tau)^{-(\frac{n}{2}+s)(\frac{n}{2}+1)}) \|v\|_{L^\infty([t_0-r^2,t_0],L^2(\Omega))} \\
			\leq& \; C r^{-\frac{n}{2}} \tau^{-\frac{n}{2}-1} (1-\tau)^{-(\frac{n}{2}+s)(\frac{n}{2}+1)}\|v\|_{L^\infty([t_0-r^2,t_0], L^2(\Omega))}
		\end{align*}
		where $C=C(n,s,t_0,\|a\|_{L^\infty(\Omega\times[0,\tilde{T}])})$, and the result is proved. \qed
	\end{proof}
	
	\section{Higher interior regularity of solutions: proof of Theorem \ref{teo.higherregularity}} \label{sec.interiorregularity}
    
    Let us prove that $u_t$ is also a Hölder continuous function. First, define the quantity
    \begin{equation*}
    	v_h(x,t) = \frac{u(x,t+h)-u(x,t)}{h}.
    \end{equation*}
    This has, for any fixed $h>0$, the same regularity properties as $u$ and it converges strongly to $u_t$ in $C([t_1,t_2], L^2(\Omega))$ for every $[t_1,t_2] \subset (0,T)$. Furthermore, it is also true that
    \begin{equation} \label{eq.vhuniformbound}
    \|v_h(\cdot,t)\|_{L^2(\Omega)}^2 \leq \frac{1}{h}\int_\Omega \int_{t}^{t+h} u_t(x,s)^2 ds dx = \frac{1}{h} \int_t^{t+h} \|u_t(\cdot,s)\|^2_{L^2(\Omega)} ds \leq \|u_t\|^2_{L^\infty([t,t+h],L^2(\Omega))},
    \end{equation}
    which means that that the $L^2(\Omega)$ norm of $v_h$ is bounded independently of $h$ for every compact subinterval $[t_1,t_2] \subset (0,T-h)$.
    
    Then, notice that $v_h$ is the solution to the following equation:
	\begin{equation} \label{1.2}
    \begin{cases}
        \displaystyle (v_h)_t(x,t)= \mathcal{L} v_h(x,t)+f^\prime(\xi(x,t))v_h(x,t), & (x,t)\in\Omega\times(0,T-h), \\
        v_h(x,0)=\frac{u(x,h)-u_0(x)}{h}, & x\in\Omega,\\
        v_h(x,t) = 0, & (x,t)\in\mathbb{R}^n \setminus \Omega \times [0,T-h),
    \end{cases}
    \end{equation}
    where $\xi(x,t)$ is a function that, for every fixed $(x,t)$, is between $u(x,t+h)$ and $u(x,t)$. The most important fact about this is that $\xi$ is bounded away from $\alpha$ independently of $h$ if we work in an interior cylinder, which means that $f^\prime(\xi(x,t))$ is an $L^\infty_{loc}([0,T-h), L^\infty(\Omega))$ function uniformly in $h$.
    
    Choose any $\tilde{T} \in (0,T-h)$ and fix some $(x_0,t_0) \in \Omega \times (0,\tilde{T})$. We notice that \eqref{1.2} is of the form \eqref{eq.zeroth.order}, taking $a(x,t) = f^\prime(\xi(x,t))$. Then we can take some $\sigma\in(0,1)$ and $R\in(0,1)$ such that $\overline{B}_R \subset \Omega$ and $[t_0-R^2,t_0] \subset (0,\tilde{T})$, and apply Theorem \ref{teo.l2.linfty} and \eqref{eq.vhuniformbound} to get
    \begin{equation} \label{eq.appliedl2linfty}
    	\|v_h\|_{L^\infty(Q_{\sigma R})} \leq C R^{-\frac{n}{2}} \sigma^{-\frac{n}{2}-1}(1-\sigma)^{-(\frac{n}{2}+s)(\frac{n}{2}+1)} \|u_t\|_{L^\infty([t_0-R^2,t_0], L^2(\Omega))},
    \end{equation}
    which is a uniform bound with respect to $h$. At the same time, \eqref{1.2} is of the form we studied in Section \ref{sec.campanato}, taking $f(v_h(x,t)) = a(x,t)v_h(x,t)$. We may then repeat the whole reasoning up to equation \eqref{eq.holder.4}. In this step, we now need to find a bound that is uniform with respect to $h$. Fix some $R \in (0,1)$ such that $\overline{B}_R \subset \Omega$ and $[t_0-R^2,t_0] \subset (0,\tilde{T})$, and take $\sigma= 1/2$. Then, using  \eqref{eq.vhuniformbound} and \eqref{eq.appliedl2linfty}, we know that for every $r< R/2$,
    \begin{align*}
    	 \Bigg( \fiint_{Q_r} & |a(x,t)|^2|v_h(x,t)|^2 dx dt \Bigg)^{1/2} \\
    	 \leq& \; \|a\|_{L^\infty(Q_r)} \|v_h\|_{L^\infty(Q_r)} \leq \|a\|_{L^\infty(\Omega \times [0,t_0])} \|v_h\|_{L^\infty(Q_{R/2})} \\
    	 \leq& \; C \|a\|_{L^\infty(\Omega \times[0,\tilde{T}])} R^{-\frac{n}{2}} \|u_t\|_{L^\infty((0,\tilde{T}),L^2(\Omega))}.
    \end{align*}
    This way we have found a bound for the reaction term average that is independent of $\tau$, $r$, the center of the cylinder $(x_0,t_0)$ and $h$. We can then proceed with the rest of the proof without issue and find that for every $\alpha_0 \in (0,1)$ and $(x_0,t_0) \in \Omega \times (0,T-h)$, there exists some $r_\ast = r_\ast(n,s,\alpha_0,\tilde{T}) \in (0,1)$ such that for every $r\leq r_\ast$,
    \begin{equation}
    	\fiint_{Q_r} |v_h(x,t)-(v_h)_{Q_r}|^2 dx dt \leq C^\prime r^{2\alpha_0},
    \end{equation}
    with $C^\prime = C^\prime (n,s,\alpha_0, r_\ast,\|a\|_{L^\infty(\Omega\times[0,\tilde{T}])},\|u_t\|_{L^2(\Omega \times (0,\tilde{T}))},\|u_t\|_{L^\infty((0,\tilde{T}),L^2(\Omega))})$. 
    
    Finally, and thanks to the fact that the constant $C^\prime$ does not depend on $h$, we conclude, taking $h\to0$ and using the strong convergence of $v_h$ to $u_t$:
    \begin{equation} \label{final.campanato.velocity}
    	\fiint_{Q_r} |u_t(x,t) -(u_t)_{Q_r}|^2 dx dt \leq C^\prime r^{2\alpha_0}.
    \end{equation}
    Then $u_t$ is locally Hölder continuous for any exponent $\alpha_0 \in (0,1)$ in $\Omega \times (0,\tilde{T})$. At the same time, thanks to Theorem \ref{teo.holder} we know that $u$ is $C^{\alpha_0}_{loc}$ and, thanks to the regularity of $f$, it is clear that $f(u)$ is $C^{\alpha_0}_{loc}$ as well. This is true for every $\tilde{T} \in (0,T)$, so $u_t \in C^{\alpha_0}_{loc}(\Omega \times (0,T))$.
    
    Finally, from \eqref{1.1} we know that 
    \begin{equation*}
    	-\Delta u(x,t) + (-\Delta)^s u(x,t) = f(u(x,t)) - u_t(x,t),
    \end{equation*}
    a.e. $x\in\Omega$ for every $t\in(0,T)$. The right hand side of this expression is in particular $C^{\alpha_0}_{loc}(\Omega)$ for every fixed $t\in(0,T)$. Then,  using \cite[Theorem 1.5]{SVWZ2025}, we conclude that $u(\cdot,t)\in C^{2,\alpha_0}_{loc}(\Omega)$ for every $t\in(0,T)$. \qed 
    
    \section*{Acknowledgements}

	Both authors were partially supported by 
	ANII (Agencia Nacional de Investigaci\'on 
	e Innovaci\'on, Uruguay) 
	grant FCE\_3\_2024\_1\_181302.
	
	S. Junquera was also partially supported by grants PID2022-137074NBI00 and PID2025-168009NB-I00 funded by
MCIU/AEI/10.13039/501100011033, Spain, and by grants FPU21/05580 and EST26/00044 from the Ministry
of Science, Innovation and Universities.
	
	L.\,M.\ Del Pezzo 
	was also partially supported 
	by grant PID2023-148028NB-I00, 
	funded by \break {MICIU/
	AEI/10.13039/501100011033 and by ERDF/EU}.

\printcredits

\bibliographystyle{cas-model2-names}

\bibliography{interior_regularity_bibliography}

\end{document}